\documentclass[11pt, a4paper,reqno]{amsart}
\usepackage[utf8]{inputenc}
\usepackage[foot]{amsaddr}
\usepackage{enumerate}
\usepackage{amsmath, amssymb,amsthm}
\usepackage{mathtools}
\usepackage{mathabx}
\usepackage{leftidx}
\numberwithin{equation}{section}
\usepackage[normalem]{ulem}
\usepackage[left=2.6cm, right=2.6cm, bottom=3cm, top=3cm]{geometry}
\usepackage{hyperref}
\usepackage{xcolor}
\definecolor{my-black}{rgb}{0,0,0}
\definecolor{my-blue}{rgb}{0,0,0.8}
\definecolor{my-red}{rgb}{0.8,0,0} 
\definecolor{my-green}{rgb}{0,0.5,0}
\hypersetup{colorlinks,
	linkcolor	= {my-blue},
	citecolor	= {my-red},
	urlcolor		= {my-black}
}
\usepackage{graphicx}
\usepackage{setspace}
\usepackage{tikz-cd}
 
\theoremstyle{plain} %italic

\newtheorem{theorem}{Theorem}
\newtheorem{lemma}[theorem]{Lemma}

\newtheorem{proposition}[theorem]{Proposition}

\newtheorem*{theorem*}{Theorem} %to avoid numbering

\theoremstyle{definition} %non-italic
\newtheorem{remark}[theorem]{Remark}

\newtheorem{claim}[theorem]{Claim}
\newtheorem{question}[theorem]{Question}

\theoremstyle{remark}
\newtheorem*{remark-non}{Remark}

\DeclareMathOperator{\supp}{supp}

\newcommand{\avgI}{\sout{1}}

\newcommand{\R}{\mathbb{R}}

\newcommand{\C}{\mathbb{C}}
\newcommand{\N}{\mathbb{N}}

\newcommand{\Z}{\mathbb{Z}}

\def\Xint#1{\mathchoice
{\XXint\displaystyle\textstyle{#1}}%
{\XXint\textstyle\scriptstyle{#1}}%
{\XXint\scriptstyle\scriptscriptstyle{#1}}%
{\XXint\scriptscriptstyle\scriptscriptstyle{#1}}%
\!\int}
\def\XXint#1#2#3{{\setbox0=\hbox{$#1{#2#3}{\int}$ }
\vcenter{\hbox{$#2#3$ }}\kern-.6\wd0}}

\def\dashint{\Xint-}

\title{On Pauli pairs and Fourier uniqueness problems}
\author[Jo\~ao P. G. Ramos]{Jo\~ao P. G. Ramos}
\address{Institute of Mathematics, EPFL, Lausanne, Switzerland}
\email{joaopgramos95@gmail.com}

\author[Mateus Sousa]{Mateus Sousa}
\address{BCAM - Basque Center for Applied Mathematics, Alameda de Mazarredo 14, 48009 Bilbao, Bizkaia,
Spain}
\email{mcosta@bcamath.org} 

\begin{document}

\begin{abstract}
We investigate the concept of Pauli pairs and a \emph{discrete} counterpart to it. In particular, we make substantial progress on the question of when a discrete Pauli pair is \emph{automatically} a classical Pauli pair. 

Effectively, if one of the functions has space and frequency Gaussian decay, and one has that $|f| = |g|$ and $|\widehat{f}| = |\widehat{g}|$ on two sets which accumulate like suitable small multiples of $\sqrt{n}$ at infinity, then $|f| \equiv |g|$ and $|\widehat{f}| = |\widehat{g}|.$ Furthermore, we show that if one drops either the assumption that one of the functions has space-frequency decay or that the discrete sets accumulate at a high rate, then the desired property no longer holds.  

Our techniques are inspired by and directly connected to several recent results in the realm of Fourier uniqueness problems \cite{kns,Radchenko-Ramos,Ramos-Sousa}, and our results may be seen as a nonlinear generalization of those. As a consequence of said techniques, we are able to prove a sharp discrete version of Hardy's uncertainty principle. 
\end{abstract}
\maketitle

\section{Introduction}

\subsection{Historical background} For $f \in L^1 \cap L^2(\R^d),$ we define its \emph{Fourier transform} as
\[
\widehat{f}(\xi) = \int_{\R^d} e^{-2\pi i \xi \cdot x} \, f(x) \, dx. 
\]
We will be interested, in this manuscript, in understanding the so-called \emph{Pauli pairs.} Indeed, two functions $f,g:\R^d \to \C$ are said to form a \emph{Pauli pair} if they satisfy that $|f| = |g|$ and $|\widehat{f}| = |\widehat{g}|$ both hold pointwise almost everywhere. In that case, $f$ and $g$ are \emph{Pauli partners}.

The main question regarding those pairs, raised by W. Pauli himself inspired by quantum mechanical considerations, is when one can guarantee that, if two functions $f$ and $g$ are Pauli partners, then $g = e^{i\alpha}f,$ for some $\alpha \in \R.$ Such a question is an instance of a \emph{phase retrieval} problem, a kind of problem which has gained massive attention in recent years due to its connections to physics, chemistry, and signal processing -- see, for instance, \cite{Rosenblatt,allman,Drenth,seelamantula,corbett,hurt,ismagilov,klibanov,millane,reichenbach} for a few of those instances, and see \cite{eldar-mendelson, candes-eldar, candes-xiaodong, alaifari-daubechies,balan,bandeira-pr,grohs-liehr,grohs-rathmair,grohs-rathmair-2,jaming-liehr,jaming-phase,hogan-lakey,Mitchell-and-friends, Christ-Mitch-Ben} and the references therein for more mathematical contributions on the subject. 

In our present context, however, this problem is \emph{ill posed}, since it is not hard to construct explicitly infinitely many functions which do not do phase retrieval. Furthermore, even if we consider solutions of the form $g = e^{i \alpha} \overline{f}$ as trivial, then the following construction still yields many non-trivial counterexamples of functions which are Pauli non-unique. 

Instead of further exploring the already rich context of Pauli pairs, the current manuscript deals with another question: given two functions $f,g:\R^d \to \C$ and two discrete sets $\Lambda, \Gamma \subset \R^d,$ if we know that 
\[
|f(\lambda)| = |g(\lambda)|, \, \forall \, \lambda \in \Lambda, \,\,\, |\widehat{f}(\gamma)| = |\widehat{g}(\gamma)|, \forall \, \gamma \in \Gamma, 
\]
then does it follow that $f$ and $g$ are Pauli partners? 

Such a problem is motivated, first and foremost, by the original problem of W. Pauli: if instead of data on the whole real line, we only have information of the probability density of the distribution of a particle on a \emph{countable} set of points, for two different times, can we guarantee that we can reconstruct the probability densities themselves in the whole space at both times? This may be understood as a version of a \emph{Fourier uniqueness} problem: if we have data on $f$ and $\widehat{f}$ on two discrete sets $\Lambda,\Gamma \subset \R^d,$ under which conditions can we ensure full reconstruction of $f$? 

The latter problem has gained a lot of attention in the past decade, since the breakthrough solution of the sphere packing problem in dimensions 8 and 24 by M. Viazovska \cite{Viazovska-8,Viazovska-24}. One of the many consequences of the breakthrough techniques used in the works of M. Viazovsa was the following \emph{Fourier interpolation formula} by D. Radchenko and M. Viazovska \cite{Radchenko-Viazovska}: there are even, real, Schwartz functions $\{a_n\}_n$ such that, if $f \in \mathcal{S}(\R)$ is an even function, then 
\[
f(x) = \sum_{n \ge 0} \left\{ f(\sqrt{n}) a_n(x) + \widehat{f}(\sqrt{n}) \widehat{a_n}(x)\right\},
\]
where the sum converges absolutely. In a related way, we highlight the results by H. Hedenmalm and A. Montes-Rodriguez  \cite{Hedenmalm-Montes-1,Hedenmalm-Montes-2,Hedenmalm-Montes-3} and the subsequent work \cite{Bakan}, which also provide Fourier uniqueness results of a similar flavour. 

In spite of the seemingly magical nature of the nodes $(\sqrt{n},\sqrt{n})_{n \ge 0},$ reconstruction results can be achieved further for way less regular sequences. The first result of such a kind seems to have been by the two authors of the present manuscript \cite{Ramos-Sousa}, where we have obtained that, if $f \in \mathcal{S}(\R)$ is such that $f(\pm n^{\alpha}) = \widehat{f}(\pm n^{\beta}) = 0,$ for some non-trivial region of indices $(\alpha,\beta),$ then $f \equiv 0.$ Developing upon these ideas, Kulikov, Nazarov and Sodin \cite{kns} extended significantly those results to the case of critical sequences. Moreover, they also constructed counterexamples in the supercritical case by employing a novel iteration method. Finally, we also highlight the other paper by the two authors of this manuscript \cite{ramos-sousa-2}, where we extend the interpolation results by D. Radchenko and M. Viazovska to the case of other critical sequences without rigid algebraic properties as in their original paper. See also \cite{ramos-stoller,stoller,radchenko-stoller,brs,ckmrv-universal,cohn-sphere,goncalves-ramos,Cohn-Goncalves,feigenbaum,Kulikov-alone,Radchenko-Ramos} for further related work. 

Beyond the importance of such results for the development of new techniques in different areas, such as that used in order to solve the sphere packing problem, we highlight that Fourier uniqueness problems are intimately connected with the deep topic of \emph{Uncertainty Principles}. Indeed, another way to formulate the results above on discrete Fourier uniqueness is the following: 
\begin{center} 
\emph{``if a function $f$ vanishes on a set $\Lambda$ and its Fourier transform vanishes on a set $\Gamma$, does it imply that $f \equiv 0$?''}
\end{center} 
In an entirely analogous way, the classical Heisenberg uncertainty principle yields that, if a function $f$ ``almost vanishes'' outside a small ball, and the same holds for its Fourier transform, then it follows that $f \equiv 0$. We refer the reader to the papers \cite{folland, Amrein-Berthier, BCK} and the references therein for further results of similar flavor. 

In that same direction, we also highlight \emph{Hardy's uncertainty principle} \cite{Hardy}, which states that the ``almost-vanishing'' above may be given in terms of Gaussian decay: if $f$ is an integrable function such that $|f(x)| \le C e^{-a \cdot \pi |x|^2}$ and $|\widehat{f}(\xi)| \le C e^{-b \cdot \pi |\xi|^2}$ hold pointwise, with $a \cdot b > 1,$ then $f \equiv 0$. This fundamental result, besides considerably extending the breadth of uncertainty inequalities, also found applications in \emph{unique continuation problems}. We refer the reader to the celebrated work of L. Escauriaza, C. Kenig, G. Ponce and L. Vega \cite{EKPV1,EKPV2} for an extension of such a result to the case of time-dependent Schr\"odinger equations. We also highlight the recent papers \cite{goncalves-ramos,kehle-ramos} as other examples of such a translation of theorems and ideas from the Fourier realm as unique continuation results, especially in the recent paradigm of discrete Fourier uniqueness. 

\subsection{Main results} The main results of this paper may be interpreted as a non-linear counterpart to the discrete Fourier uniqueness problems mentioned above. Indeed, we will be first concerned with the one-dimensional setting, where we wish to show that information on the absolute values $|f|$ and $|\widehat{f}|$ on a discrete set implies full knowledge of the respective absolute values. By the non-linear nature of this problem, it differs significantly from the linear one and, in general, we will see that such a recovery may \emph{never} be achieved if we do not suppose additional constraint on at least one of the functions involved. This is the content of our first result, which may be interpreted as a \emph{negative} answer to the general case of the main question we aim to address at this paper. 

\begin{theorem}\label{thm:negative} Let $\Lambda, \Gamma \subset \R$ be two discrete sets. Then there are two functions $f,g:\R \to \C$ with $|f| \neq |g|$ and $|\widehat{f}| \neq |\widehat{g}|,$ such that $|f(\lambda)| = |g(\lambda)|, \, \forall \, \lambda \in \Lambda,$ and $|\widehat{f}(\gamma)| = |\widehat{g}(\gamma)|, \, \forall \, \gamma \in \Gamma.$ 
\end{theorem}

As we shall see, the construction used in order to prove Theorem \ref{thm:negative} is simple, and uses heavily the fact that we do not impose any structure on the functions in its statement. The main theorem in this paper is, on the other hand, an answer to what happens if one imposes that at least one of the functions $f,g$ has robust decay - so much that some additional \emph{global} structure can be inferred from \emph{discrete} decay on both space and frequency. 

Before stating that result, we note that we shall write $\widehat{f} =\mathcal{F}(f)$ to denote the Fourier transform of the functions $f$ throughout the manuscript. Moreover, we let 
\[
H^s = \left\{ f:\R \to \C \colon \int_{\R} |\widehat{f}(\xi)|^2 (1+|\xi|^2)^{s/2} \, d \xi < +\infty \right\}
\]
denote the $L^2-$based Sobolev space of order $s$ on the real line. In that regard, we will use the notation 
$$\mathcal{F}(H^1) = \{ f: \R \to \C \colon \widecheck{f} \in H^1\}.$$

\begin{theorem}\label{thm:positive} Let $\Lambda, \Gamma$ denote two discrete subsets of $\R.$ The following assertions hold: 
\begin{enumerate}[\normalfont(I)]
\item Fix any $\alpha > 0.$ There is $C_{\alpha} > 0$ such that the following holds. Suppose $\Lambda, \Gamma$ satisfy that
\begin{equation}\label{eq:sets}
\max\left\{\limsup_{i\to \pm \infty} |\lambda_{i+1} - \lambda_i| |\lambda_i|, \limsup_{i\to \pm \infty} |\gamma_{i+1} - \gamma_i| |\gamma_i|\right\} < \frac{1}{C_{\alpha}}.
\end{equation}
 Let $f,g:\R \to \C$ be two functions in $H^1 \cap \mathcal{F}(H^1).$ Suppose, moreover, that the function $g$ satisfies 
\[
|g(x)| + |\widehat{g}(x)| \le Ce^{-\alpha x^2},
\]
and that we have $|f(\lambda)| = |g(\lambda)|, \, \forall \, \lambda \in \Lambda, \, \, |\widehat{f}(\gamma)| = |\widehat{g}(\gamma)|, \, \forall \, \gamma \in \Gamma.$ Then it follows that $|f| \equiv |g|$ and $|\widehat{f}| \equiv |\widehat{g}|.$

\item Conversely, suppose that the sets $\Lambda, \Gamma$ satisfy the following \emph{supercritical density condition}: 
\begin{equation}\label{eq:weak-density}
\min\left\{\liminf_{i\to \pm \infty} |\lambda_{i+1} - \lambda_i| |\lambda_i\, \liminf_{i\to \pm \infty} |\gamma_{i+1} - \gamma_i| |\gamma_i|\right\} > \frac{1}{2}. 
\end{equation} 
Then there is $a_0 >0$ with the followifng property: for each $\alpha \in (0,a_0),$ there is an infinite-dimensional space of functions $E_{\alpha} \subset \mathcal{S}$ such that, if $f,g \in E,$ then 
\[
|f(x)| + |g(x)| + |\widehat{f}(x)| + |\widehat{g}(x)| \le C e^{-\alpha |x|^2},
\]
and $|f(\lambda)| = |g(\lambda)|, \, \forall \lambda \in \Lambda,$ as well as $|\widehat{f}(\gamma)| = |\widehat{g}(\gamma)|,\, \forall \gamma \in \Gamma,$ but it does \emph{not} hold that either $|f|=|g|$ or $|\widehat{f}| = |\widehat{g}|$. 
\end{enumerate}
\end{theorem}

The proof of Theorem \ref{thm:positive} is divided into some steps. First, we show that, if $f$ has Gaussian decay at the nodes of a set $\Lambda$ and $\widehat{f}$ has Gaussian decay at the nodes of a set $\Gamma,$ where $\Lambda, \Gamma$ are as in Theorem \ref{thm:positive}, then $f$ and $\widehat{f}$ possess Gaussian decay \emph{everywhere}. The proof of this fact is an adaptation of the techniques recently developed in \cite{kns}, which in turn are based on an optimization argument run together with a suitable Poincar\'e--Wirtinger inequality. We also mention the recent paper \cite{Radchenko-Ramos} as another instance where such techniques have been successfully used. 

After one obtains the desired Gaussian decay of $f$ and $\widehat{f},$ we use the classical Paley-Wiener theorem, which states in one of its most general versions that if a function decays super exponentially, then its Fourier transform may be extended as an entire function of a given order. This then implies that $f,g,\widehat{f},\widehat{g}$ may all be extended to be entire functions of order 2, with an explicit control on the type of each of them. 

The last part of the argument is simply an utilization of the function $H = f \cdot f^* - g \cdot g^*,$ where $f^*,g^*$ denote the holomorphic conjugates of $f,g,$ respectively. Our conditions show that $H$ vanishes on the set $\Lambda.$ Since $H$ is of order 2 and finite type, as long as $\Lambda$ is dense enough, we should have by analytic continuation that $H$ is identically zero. Since the same can be used for the respective Fourier transforms, we conclude the statement of Theorem \ref{thm:positive}. 

The last result we highlight in this manuscript is an application of the main techniques and ideas developed here, dealing with \emph{discrete} version of Hardy's uncertainty principle: 

\begin{theorem}\label{thm:discrete-hardy} Let $A>1$ be fixed. Then there is $C_A > 0$ such that the following holds. Let $\Lambda,\Gamma$ be two discrete sets, such that 
\begin{equation}\label{eq:upper-bound-density-hardy}
\max\left( \limsup_{j \to \pm \infty}  |\lambda_j - \lambda_{j+1}| |\lambda_j|, \limsup_{j \to \pm \infty} |\gamma_j - \gamma_{j+1}||\gamma_j| \right) < \frac{1}{C_A}.  
\end{equation}
Suppose that a function $f \in H^1\cap\mathcal{F}(H^1)$ is such that 
\[
|f(\lambda)| \le C \cdot e^{-A \pi |\lambda|^2},\,\forall \, \lambda \in \Lambda, \,\,|\widehat{f}(\gamma)| \le C \cdot e^{-A \pi |\gamma|^2},\,\forall \, \gamma \in \Gamma.
\]
Then $f \equiv 0.$
\end{theorem}

Much as in the proof of Theorem \ref{thm:positive}, the proof of Theorem \ref{thm:discrete-hardy} goes through an explicitly quantitative version of Proposition \ref{prop:gaussian}. We then make thorough use of the estimates obtained in order to conclude that any functions $f$ satisfying the hypotheses of Theorem \ref{thm:discrete-hardy} has $L^2$ Gaussian decay of a sufficiently high order. This, together with the use of the Cowling-Price uncertainty principle \cite{Cowling-Price}, allows us to conclude the proof of that result.

Finally, we remark that, in light of the recent paper by D. Radchenko and the first author \cite{Radchenko-Ramos}, one is able to achieve suitable higher-dimensional generalizations of the main results in this manuscript. These results read as follows:

\begin{theorem}\label{thm:high-d-gen-pauli}Let $\Lambda, \Gamma$ denote two discrete subsets of $\R_+.$ The following assertions hold: 

\begin{enumerate}[\normalfont(I)]
\item Fix any $\alpha > 0.$ There is $C_{\alpha} > 0$ such that the following holds. Suppose $\Lambda, \Gamma$ satisfy \eqref{eq:sets}. Let $f,g:\R^d \to \C$ be two functions in $H^1 \cap \mathcal{F}(H^1).$ Suppose, moreover, that the function $g$ satisfies 
\[
|g(x)| + |\widehat{g}(x)| \le Ce^{-\alpha |x|^2},
\]
and that we have $|f(\lambda \cdot u)| = |g(\lambda \cdot u)|, \, \forall \, \lambda \in \Lambda, \, \forall \, u \in \mathbb{S}^{d-1}, \, \, |\widehat{f}(\gamma \cdot u)| = |\widehat{g}(\gamma \cdot u)|, \, \forall \, \gamma \in \Gamma, \, \forall \, u \in \mathbb{S}^{d-1}.$ Then it follows that $|f| \equiv |g|$ and $|\widehat{f}| \equiv |\widehat{g}|.$

\item Conversely, suppose that the sets $\Lambda, \Gamma$ satisfy \eqref{eq:weak-density}. 
Then there is $a_0 >0$ with the following property: for each $\alpha \in (0,a_0),$ there is an infinite-dimensional space of functions $E_{\alpha} \subset \mathcal{S}$ such that, if $f,g \in E,$ then 
\[
|f(x)| + |g(x)| + |\widehat{f}(x)| + |\widehat{g}(x)| \le C e^{-\alpha |x|^2},
\]
and $|f(\lambda \cdot u)| = |g(\lambda \cdot u)|, \, \forall \lambda \in \Lambda, \, \forall \, u \in \mathbb{S}^{d-1},$ as well as $|\widehat{f}(\gamma \cdot u)| = |\widehat{g}(\gamma \cdot u)|,\, \forall \gamma \in \Gamma,\, \forall \, u \in \mathbb{S}^{d-1},$ but it does \emph{not} hold that either $|f|=|g|$ or $|\widehat{f}| = |\widehat{g}|$. 
\end{enumerate}
\end{theorem}

\begin{theorem}\label{thm:discrete-hardy-high-d} Let $A>1$ be fixed. Then there is $C_A > 0$ (which may depend on the dimension $d$) such that the following holds. Let $\Lambda,\Gamma \subset \R_+$ be two discrete sets satisfying \eqref{eq:upper-bound-density-hardy}.
Suppose that a function $f \in H^1\cap\mathcal{F}(H^1)$ is such that 
\[
|f(\lambda \cdot u)| \le C \cdot e^{-A \pi |\lambda|^2},\,\forall \, \lambda \in \Lambda, \, \forall \, u \in \mathbb{S}^{d-1}, \,  \,\,|\widehat{f}(\gamma \cdot u)| \le C \cdot e^{-A \pi |\gamma|^2},\,\forall \, \gamma \in \Gamma, \, \forall \, u \in \mathbb{S}^{d-1}. 
\]
Then $f \equiv 0.$
\end{theorem}

This paper is organized as follows. First, we prove Theorems \ref{thm:negative} and \ref{thm:positive} in Section 1. We subsequently present the proof of Theorem \ref{thm:discrete-hardy}, and move on to prove the higher-dimensional results (Theorems \ref{thm:high-d-gen-pauli} and \ref{thm:discrete-hardy-high-d}) in the Section thereafter. Finally, in the final section of this manuscript, we shall comment on possible generalizations and open problems stemming from the main results and consideration of this paper. 

\subsection*{Acknowledgements} The authors are thankful to Mitchell Taylor for conversations which encouraged us to investigate the results in this manuscript. MS is supported by grants RYC2022-038226-I and PID2020-113156GB-I00 funded by MICIU/AEI/10.13039/501100011033 and by ESF+, the Basque Government through the BERC 2022-2025 program, and through BCAM Severo Ochoa accreditation CEX2021-001142-S / MICIN / AEI / 10.13039/501100011033. 

\section{Proof of Theorems \ref{thm:negative} and \ref{thm:positive}}

\subsection{Proof of Theorem \ref{thm:negative}} Let us quickly show that, without restrictions on either the sets $\Lambda, \Gamma$ or on at least one of the functions $f,g,$ global Pauli partners cannot be inferred from local ones. \\

\begin{proof}[Proof of Theorem \ref{thm:negative}] Let $\Lambda, \Gamma$ be two arbitrary discrete sets on $\R.$ Let $I \subset \R$ be an interval such that $I \cap \Lambda = \emptyset.$ Let $g$ be a smooth function, such that $\text{supp}(g) \subset I.$ 

We then take $h:\R \to \R$ any smooth function such that $|\widehat{h}| \not \equiv |\widehat{h+g}|$ and $\widehat{h}|_{\Gamma} = - \frac{1}{2}\widehat{g}|_{\Gamma}.$ This function exists trivially: consider, for instance, $-\frac{1}{2}\cdot g + \sum_{i \ge 0} \psi_i,$ where $\supp(\widehat{\psi_i}) \cap \Gamma = \emptyset.$ We claim that $h+g$ and $h$ satisfy that $|h+g| = |h|$ on $\Lambda$ and $|\widehat{h+g}| = |\widehat{h}|$ on $\Gamma,$ but $|h+g| \neq |h|$ and $|\widehat{h+g}| \neq |\widehat{h}|.$ 

Indeed, since the support of $g$ does not intersect $\Lambda,$ the assertion that $|h+g| = |h|$ on $\Lambda$ is trivially true. Moreover, since $\widehat{h} = - \frac{1}{2} \widehat{g}$ on $\Gamma,$ we have by the previous considerations $|\widehat{h+g}| = |\widehat{h}|$ on $\Gamma$.
Furthermore, since $g \not\equiv 0,$ we have that $|h+g| \not\equiv |h|,$ and the same holds for the last assertion. 
\end{proof}

\begin{remark}
Notice that, as non-intuitive as it might seem, the construction above does not contradict in any way the positive result Theorem \ref{thm:positive}. Indeed, two main aspects are to be noted in the construction above: first of all, we have not assumed anything on the functions involved besides the relations between them needed in order to run the construction. As a matter of fact, the function $h$ is taken such that it interpolates, on the Fourier side, the function $-g/2$ at specific nodes. 

Suppose then that the sets $\Lambda$ and $\Gamma$ and the function $h$ satisfy the conditions in of Theorem \ref{thm:positive}. Evidently, the function $g$ constructed above must satisfy Gaussian decay at the nodes of $\Lambda$, and its Fourier transform must satisfy Gaussian decay at the nodes of $\Gamma.$ By Proposition \ref{prop:gaussian} below, this implies that $g$ itself has Gaussian decay both in space and in frequency. By Paley-Wiener, $g$ can be extended as an analytic function on the whole complex plane, which, by the fact that $g$ is compactly supported, implies that $g \equiv 0.$ 

Note that we may even relax the hypothesis that $g$ is compactly supported for this argument: Proposition \ref{prop:gaussian} still implies that $g$ has Gaussian decay in space and in frequency, and since by the proof of Theorem \ref{thm:positive} both $\Lambda$ and $\Gamma$ are taken to be uniqueness sets for analytic functions, the hypotheses still imply that $g \equiv 0$, as desired. 
\end{remark}

\subsection{Proof of Theorem \ref{thm:positive}} We start with the proof of Part (I) of Theorem \ref{thm:positive}. We first prove that, if a function $f$ and its Fourier transform $\widehat{f}$ both have Gaussian decay at specific discrete sets of points, then the Gaussian decay may be extended to the full line. 

\begin{proposition}\label{prop:gaussian} Suppose $\Lambda=\{\lambda_i\}_{i \in \Z}, \Gamma=\{\gamma_i\}_{i \in \Z}$ are such that
\begin{equation}
\max\left( \limsup_{i \to \pm \infty} |\lambda_{i+1} - \lambda_i| |\lambda_i|, \limsup_{j \to \pm \infty} |\gamma_{j+1} - \gamma_j| |\gamma_j|\right) < \frac{1}{2}.
\end{equation}
 Suppose that a function $f:\R \to \mathbb{C}$ belongs to $H^1 \cap \mathcal{F}(H^1)$ and satisfies that 
\[
|f(\lambda)| \le C e^{-a |\lambda|^2}, \, \forall \lambda \in \Lambda, \,\, |\widehat{f}(\gamma)| \le C e^{-a|\gamma|^2},\, \, \forall \gamma \in \Gamma. 
\]
Then there is $c_a > 0$ Such that 
\[
|f(x)| + |\widehat{f}(x)| \le C e^{-c_a |x|^2}, \, \forall x \in \R. 
\]
\end{proposition}

In order to prove this proposition, we will need a preliminary result, which is an alternative form of Wirtinger's inequality. We remark that this present version is taken from Lemma 1 in \cite{kns}, and we refer the reader to that paper for its proof. 

\begin{lemma}\label{lemma:wirtinger} Suppose $f:\R \to \C$ is a $H^1$ function. Then for each $\varepsilon > 0,$ we have 
\begin{equation}\label{eq:wirtinger}
\int_{I} |f(x)|^2 \, dx \le \pi^{-2} |I|^2 (1+\varepsilon) \int_I |f'|^2 + (1+\varepsilon^{-1})|I| \left( |f(a)|^2 + |f(b)|^2\right).
\end{equation}
\end{lemma}

%\begin{proof} By translation and scaling invariance, we may assume without loss of generality that $a = 0, b=1.$ We simply employ Wirtinger's inequality: this shows that 

%\begin{align}\label{eq:Wirt}
%& \int_{I}|f(x)-\ell(t)|^{2} \leqslant \pi^{-2}\int_{I}\left|f^{\prime}-\ell'\right|^{2}.
%\end{align}
%Here, we define $\ell$ as the unique affine function such that 
%$$
%\begin{aligned} 
%&\ell(a)=f(a),\ell(b)=f(b).
%\end{aligned} 
%$$
%This implies that 

%$$ (f(b)-f(a)) t+c=\ell(t),$$ 
%which yields $c=f(a).$ We then use triangle's inequality in \eqref{eq:Wirt}, both to bound the left-hand side from below and the right-hand side from above, together with noting that $|\ell(t)| \le \max(|f(a)|,|f(b)|)$ pointwise. Since 
%\[
%\int_I |f'|^2 \, dt = \int_{I} |(f'-\ell') + \ell'|^2 \, dt = %\int_I |f'-\ell'|^2 + \int_I |\ell'|^2 \ge \int_{I}|f'-\ell'|^2,
%\]
%which follows from the fact that $\int_I f' = \int_I \ell',$ together with the fact that $\ell'$ is identically constant, then it is not hard to obtain, for any $\varepsilon>0,$ \eqref{eq:wirtinger}. This concludes the proof. 
%\end{proof} 

This inequality will be of particular importance to us for the next result, which is a rewriting of \cite[Lemma~3.(i)]{kns}. For convenience, we provide a proof of it below. 

\begin{proposition}\label{prop:convex} Suppose that $f:\R \to \C$ is a Schwartz function which vanishes on a set $\tilde{\Lambda},$ such that $\R \setminus \tilde{\Lambda}$ is composed of intervals of length at most $(1-\varepsilon)\cdot (2\mu)^{-1}.$ Then, for any $C^1,$ increasing, convex function $\Phi:[0,\infty) \to \R$, we have 
\begin{align}\label{eq:conclusion-wirtinger}
\Phi(\mu^2) \int_{\R}|f(y)|^2 \, dy \le  \int_{\R} \Phi(|y|^2) |\widehat{f}(y)|^2 \, dy + C_{\varepsilon} \mu \cdot \Phi'(\mu^2) \sum_{\lambda \in \tilde{\Lambda}} |f(\lambda)|^2 
\end{align}
\end{proposition}

\begin{proof} By Lemma \ref{lemma:wirtinger}, we obtain that 
\[
\int_I |f(x)|^2 \, dx \le \pi^{-2} (2\mu)^{-2} \int_I |f'|^2 + (1+\varepsilon^{-1})(1-\varepsilon)(2\mu)^{-1} \left(|f(a)|^2 + |f(b)|^2\right),
\]
which when summed over $I$ an interval from $\R \setminus \tilde{\Lambda}$ yields, after Plancherel, 
\begin{equation}\label{eq:conseq-poinc}
\int_{\R} |f(x)|^2 \, dx \le \mu^{-2} \int_{\R}|y|^2 |\widehat{f}(y)|^2 \, dy + C_{\varepsilon} \mu^{-1} \sum_{\lambda \in \Lambda} |f(\lambda)|^2. 
\end{equation}
%Before going on, we note that $C_\varepsilon$ is nothing but an absolute constant times $1+\varepsilon^{-1}.$ We also note that, from now on, $\varepsilon$ will denote the \emph{fixed quantity} 
%$\varepsilon = \mu^{1/2} \frac{ \left( \sum_{\lambda \in \Lambda} |f(\lambda)|^2 \right)^{1/2}}{\left( \int_{\R} |y|^2 |\widehat{f}(y)|^2 \, dy \right)^{1/2}}$ . 
Since the inequalities derived so far hold true for \emph{any} $\varepsilon > 0,$ there is no issue in selecting such value. We then rewrite \eqref{eq:conseq-poinc} as 
$$
\begin{aligned}
& \int_{|y| \geq \mu} \left( \left( \frac{|y|}{\mu} \right)^2 - 1\right) |\widehat{f}(y)|^2 \, dy +C_{\varepsilon} (\mu)^{-1} \sum_{\lambda \in \Lambda}|f(\lambda)|^{2} \\
& \geq \int_{|y| \leq \mu}\left(-\left( \frac{|y|}{\mu}\right)^{2}+1\right)|\widehat{f}(y)|^{2} dy. 
\end{aligned}
$$

Multiplying the inequality above by $\Phi'(\mu^2)$, using that $\Phi(y^2) - \Phi(\mu^2) \ge \Phi'(\mu^2)(y^2 - \mu^2)$ for $|y| \ge \mu,$ and that the reverse inequality holds for $|y| \le \mu,$ we have  
\[
\int_{|y| \ge \mu}   \left( \Phi(y^2) - \Phi(\mu^2) \right) |\widehat{f}(y)|^2 \, dy + C_{\varepsilon} \mu \cdot \Phi'(\mu^2) \sum_{\lambda \in \Lambda}|f(\lambda)|^{2} 
\]
\[
\geq \int_{|y| \leq \mu}\left(\Phi(\mu^2) - \Phi(y^2)\right)|\widehat{f}(y)|^{2} dy. 
\]
Upon reordering, this shows \eqref{eq:conclusion-wirtinger}, as desired. 
\end{proof}
We next note the following estimate on the points of $\Lambda$ and $\Gamma$, which will be useful throughout the manuscript. 

\begin{lemma}\label{lemma:decay-lambda} Let $\Lambda,\Gamma$ be as in the statement of Theorem \ref{thm:positive}, and let $\varepsilon>0$ be arbitrary. Then we have that 
\[
|\gamma_i| + |\lambda_i| \le \sqrt{\frac{2+\varepsilon}{C_{\alpha}}} \sqrt{i},
\]
for all $i \in \Z$ sufficiently large. 
\end{lemma}

\begin{proof} Fix $\varepsilon > 0.$ For $i$ sufficiently large, we have that $|\lambda_i| > 1/\varepsilon$ and $|\lambda_{i+1} - \lambda_i| < \varepsilon,$ which implies that $|\lambda_{i+1}/\lambda_i| <1+ \varepsilon^2$ for all such $i \in \Z$. By the hypotheses of Theorem \ref{thm:positive}, we have that 
\[
|\lambda_{i+1}^2- \lambda_i^2| = 2 \left| \int_{\lambda_i}^{\lambda_{i+1}} t^2 \, dt \right| \le 2(1+\varepsilon^2)^2 |\lambda_{i+1}-\lambda_i| |\lambda_i|^2 < \frac{2+\varepsilon}{C_{\alpha}}, 
\]
with $\varepsilon > 0$ as small as one wants, if one makes $i \in \N$ larger. Telescoping then yields that 
\[
|\lambda_i|^2 \le \frac{2+\varepsilon}{C_{\alpha}} i,
\]
finishing our claim for $\Lambda$. The conclusion for $\Gamma$ is entirely analogous. 
\end{proof}

We are now ready to prove Proposition \ref{prop:gaussian}. The first step in that direction is, first and foremost, showing that each function $f \in H^1 \cap \mathcal{F}(H^1)$ satisfying the conditions of Proposition \ref{prop:gaussian} is actually smooth. We first need the following result about the sets $\Lambda,\Gamma$. We refer the reader to \cite[Claim~1]{kns} for a proof. 

\begin{lemma}\label{lemma:sets-lower} Let $\Lambda,\Gamma$ be as in Proposition \ref{prop:gaussian}. Then there are $\delta > 0$ and $\{\lambda_i'\}_i = \Lambda' \subset \Lambda, \, \{\gamma_i'\}_i = \Gamma' \subset \Gamma$ such that 
\begin{align}
\label{eq:upp-low-1} \frac{1}{C_{\alpha}} &> \limsup_{n \to \pm \infty} |\lambda_{i+1}' - \lambda_i'| |\lambda_i'| \ge \liminf_{n\to \pm \infty}|\lambda_{i+1}' - \lambda_i'||\lambda_i'| \ge \delta, \\ 
\label{eq:upp-low-2}  \frac{1}{C_{\alpha}} &> \limsup_{n \to \pm \infty} |\gamma_{i+1}' - \gamma_i'| |\gamma_i'| \ge \liminf_{n\to \pm \infty}|\gamma_{i+1}' - \gamma_i'||\gamma_i'| \ge \delta. 
\end{align}
\end{lemma} 

We hence replace $\Lambda, \Gamma$ by the sets $\Lambda', \Gamma'$, and note the following consequence of \cite{kns}: 

\begin{proposition}[Corollary~1 in \cite{kns}]\label{prop:schwartz-dec} Under the same hypotheses on $\Lambda, \Gamma$ as in Proposition \ref{prop:gaussian}, we have that $f \in H^1 \cap \mathcal{F}(H^1)$ if, and only if, for each $r>0$, we have
\[
|f(\lambda_j)| = O(|\lambda_j|^{-r}), \, \forall \, j \in \Z, \, |\widehat{f}(\gamma_i)| = O(|\gamma_i|^{-r}), \, \forall \, i \in \Z. 
\]
\end{proposition}

That result directly implies the following:

\begin{proposition}\label{prop:Schwartz} Under the hypotheses of Proposition \ref{prop:gaussian}, we have $f \in \mathcal{S}(\R).$
\end{proposition}

We are now finally able to finish the proof of Proposition \ref{prop:gaussian}. The argument that follows is a quantified version of \cite[Eq.~(6.1.1)]{kns}. 

\begin{proof}[Proof of Proposition \ref{prop:gaussian}]  We work under the same hypotheses on the sets of the proof of Proposition \ref{prop:Schwartz}. Fix then $u >0$. And consider the function $F$ defined by the $(k+1)$-fold convolution 

\begin{equation}\label{eq:F-special}
F(t) = 1_{[-u,u]} * \left(\frac{k}{u} 1_{[-\frac{u}{2k},\frac{u}{2k}]}\right) * \cdots * \left(\frac{k}{u} 1_{[-\frac{u}{2k},\frac{u}{2k}]}\right) (t).
\end{equation}
Consider $F_{\xi}(t) = F(t-\xi),$ as in \cite{kns}. We then use the Proposition \ref{prop:gaussian} for the function $f \cdot F_{\xi}.$ This function satisfies the conditions of such a proposition with $\mu = \sigma \left( |\xi| - \frac{3}{2} u\right),$ and we obtain 

\begin{align}\label{eq:prelim-bound-1}
 \int_{\mathbb{R}}\left(|\xi|-\frac{3}{2} u\right)^{2p}|f|^{2}\left|F_{\xi}\right|^{2} \leqslant \sigma^{-2 p} \int_{\mathbb{R}}|y|^{2p}\left|\widehat{f} * \widehat{F_{\xi}}\right|^{2} 
 \end{align}
 \begin{align*} 
+ p \left( |\xi| - \frac{3}{2} u\right)^{2p - 1} \sigma^{-1} \sum_{\lambda \in \Lambda} |f(\lambda)|^2 |F_{\xi}(\lambda)|^2. 
\end{align*}

We now integrate the expression above over $|\xi| \geqslant \frac{3}{2} u+X_{0}$: the left-hand side and the first term on the right-hand side are bounded as in \cite{kns}. Effectively, for the term on the left-hand side, by the fact that if $y \in \operatorname{supp}\left(F_{\xi}\right) \Rightarrow|y-\xi| \leqslant \frac{3 u}{2} \Rightarrow|\xi|-\frac{3 u}{2} \ge |y|-3u$, if we fix $K>1$ and $X_0$ is sufficiently large, we have 
\[
\int_{|\xi| \frac{3}{2} u + X_0} \int_{\mathbb{R}}\left(|\xi|-\frac{3}{2} u\right)^{2p}|f|^{2}\left|F_{\xi}\right|^{2} \, dx \, d\xi \ge u \cdot \left(\frac{K-3}{K}\right)^{2p} \int_{|x|\ge K \cdot u} |x|^{2p} |f(x)|^2 \, dx. 
\]
For the first term on the right-hand side, we have trivially by Plancherel that it is bounded by 
\[
\int_{\R} \int_{\R} |\widehat{f}(\eta)|^2 |\widehat{F}(\eta)|^2 |\zeta + \eta|^2 \, d \zeta \, d \eta,
\]
which is, again by repeating the same argument as in \cite[Section~5.2]{kns}, bounded from above by 
\[
2 u \sigma^{-2p} \int_{\R} |\widehat{f}|^2 \left( |y| + \frac{k}{\pi u}\right)^{2p} \, dy,
\]
as long as the parameter $k \ge p.$ Finally, for the second term on the right-hand side of \eqref{eq:prelim-bound-1}, we first bound it by using that if $\lambda \in \text{supp}(F_\xi),$ then $|\lambda| \ge |\xi| - \frac{3}{2}u$. This shows that that factor is bounded by 
\[
C u \cdot \sum_{\lambda \in \Lambda} |\lambda|^{2p-1} |f(\lambda)|^2 
\]
Putting all estimates together, we have that, for $K>4,$ $ k \in (p,\pi \cdot p),$ and $u \geqslant X_{0}$,

\begin{align*}
& \left( \frac{K-3}{K}\right)^{2 p}\left(\int_{|x| \geqslant K \cdot u}|x|^{2 p}|f|^{2}\right) \\
& \leqslant \sigma^{-2p} \cdot \int_{\mathbb{R}}|\widehat{f}|^{2}\left(|y|+\frac{k}{\pi u}\right)^{2 p} d y + C \cdot p \sigma^{-1} \cdot \sum_{\lambda \in \Lambda} |\lambda|^{2p-1} |f(\lambda)|^2.
\end{align*}
This plainly implies that 

\begin{align*}
\int|x|^{2 p}|f|^{2} \leq \left(\frac{K-1}{K}\left(\frac{K}{K-3}\right) \sigma^{-1}\right)^{2 p} \cdot 2 \int_{\mathbb{R}} |\widehat{f}(y)|^{2}|y|^{2p} d y 
\end{align*}
\begin{align*} 
 +\int_{|x| \leq K \cdot u}|x|^{2 p} \cdot|f|^{2}+2 \int_{|y| \leqslant K \cdot \frac{k}{u}}|\widehat{f}|^{2}\left(|y|+\frac{k}{\pi u}\right)^{2 p} 
\end{align*}
\begin{align*} 
+ C \cdot p \sigma^{-1} \cdot \left( \frac{K}{K-3}\right)^{2p} \sum_{\lambda \in \Lambda} |\lambda|^{2p-1} |f(\lambda)|^2 
\end{align*}
\begin{align*}
& \leqslant \frac{1}{2} \int_{\R}|\widehat{f}(y)|^{2}|y|^{2 p} d y 
+(K u)^{2 p} \int|f|^{2}+2(k+1)^{2 p}\left(\frac{p}{u}\right)^{2 p} \int|f|^{2} \cr
& + C \cdot p \sigma^{-1} \cdot \left( \frac{K}{K-3}\right)^{2p} \sum_{\lambda \in \Lambda} |\lambda|^{2p-1} |f(\lambda)|^2.
\end{align*}

We now let $u=p^{\frac{1}{2}}$. Since $\sigma > 1$, if $K>4$ is made sufficiently large, we then get 

\begin{align*} 
\int_{\mathbb{R}}|x|^{2 p} |f|^{2} & \leqslant \frac{1}{2} \int_{\mathbb{R}}|y|^{2 p}|\widehat{f}|^{2}+C^{2 p} p^{ p} \int_{\mathbb{R}}|f|^{2} + C \cdot (1+\varepsilon)^p \sum_{\lambda \in \Lambda} |\lambda|^{2p-1} |f(\lambda)|^2,
\end{align*} 
where $\varepsilon>0$ is a fixed small constant. If we do the same for $\widehat{f}$, and use the same strategy as in \cite{kns}, we obtain 
\begin{align*}
\int_{\R} |y|^{2p} |\widehat{f}|^2 & \le \frac{1}{2} \int_{\R} |x|^{2p} |f|^2 + C^{2p} p^{p} \int_{\R} |f|^2 + C \cdot (1+\varepsilon)^p \sum_{\gamma \in \Gamma} |\gamma|^{2p - 1} |\widehat{f}(\gamma)|^2. 
\end{align*}
This implies that 
$$
\begin{aligned}
& \int_{\mathbb{R}}|x|^{2 p}|f(x)|^{2} d x \leqslant C^{p} p^{ p} \int_{\mathbb{R}}|f|^{2} + C \cdot (1+\varepsilon)^p \left( \sum_{\lambda \in \Lambda} |\lambda|^{2p - 1}|f(\lambda)|^2 + \sum_{\gamma \in \Gamma} |\gamma|^{2p - 1} |\widehat{f}(\gamma)|^2\right). 
\end{aligned}
$$

As $f$ and $\widehat{f}$ have Gaussian decay at the nodes of $\Lambda$ and $\Gamma$, we may use the following Riemann sum strategy to bound the sums on the right-hand side:

\begin{align}\label{eq:upper-bound-discrete}
\sum_{\lambda \in \Lambda}|\lambda|^{2p-1}|f(\lambda)|^{2} & \le \sum_{\lambda \in \Lambda} |\lambda|^{2p-1} e^{-2\alpha |\lambda|^2} \lesssim C_{\alpha} \sum_{i \in \Z} |\lambda_i|^{2p} |\lambda_{i+1} - \lambda_i| e^{-2\alpha|\lambda_i|^2} \cr 
 & \lesssim \int_{\R} |x|^{2p} e^{-2\alpha|x|^2} \, dx \lesssim \rho^p \Gamma(p+1),
\end{align}
for some $\rho>0$. Analogously, we also get that
\begin{align*}
\sum_{\gamma \in \Gamma} |\gamma|^{2p - 1} |\widehat{f}(\gamma)|^2 & \lesssim \rho^p \Gamma(p+1). 
\end{align*}
As a consequence of the argument above and Stirling's formula, we get that, for some $C>1$,
$$
\int_{\R} |\xi|^{2p} |\widehat{f}(\xi)|^2 \, d\xi + \int_{\mathbb{R}}|x|^{2 p}|f(x)|^{2} d x \leqslant C^{p} p^{p} \cdot \widetilde{C}_{f}.
$$
Optimising in $p > 0$, this shows that 

$$
\int_{\R} e^{\beta |\xi|^2} |\widehat{f}(\xi)|^2 \, d\xi + \int_{\mathbb{R}} e^{\beta |x|^{2}}|f(x)|^{2} d x \leqslant C_{f} \text {, for some } \beta >0. 
$$
 Hence, we have proved that Gaussian decay from the nodes propagates to the whole space, as desired, finishing thus the proof of Proposition \ref{prop:gaussian}. \end{proof}

 We are now able to finish the proof of Part (I) of Theorem \ref{thm:positive}. 

\begin{proof}[Proof of Theorem \ref{thm:positive}]
Suppose now that two functions $f,g$ satisfy the conditions of Theorem \ref{thm:positive}. Suppose then that, for two sets $\Lambda$ and $\Gamma$ as in the statement of that result, we have 
\[
|f(\lambda)| = |g(\lambda)|, \,\, |\widehat{f}(\gamma)| = | \widehat{g}(\gamma)|, \forall \lambda \in \Lambda, \, \forall \gamma \in \Gamma. 
\]
By Proposition \ref{prop:gaussian}, this directly implies Gaussian decay for $f$, both in space and frequency. Now, we invoke the following version of the Paley-Wiener theorem: 

\begin{proposition}\label{prop:paley-wiener} Suppose that $f \in L^2$ satisfies $|f(x)| \le C e^{-\alpha|x|^2}$ pointwise almost everywhere, for some $\alpha>0.$ Then It follows that $\widehat{f}$ may be extended to the whole complex plane as an entire function. Moreover, there is $C(\alpha) > 0$ such that 
\begin{equation}\label{eq:paley-wiener}
|\widehat{f}(z)| \le C \cdot e^{\theta(\alpha) |z|^2}, \, z \in \C,
\end{equation}
where we may take $\theta(\alpha) = \frac{\pi^2}{\alpha}.$
\end{proposition}

\begin{proof} The proof of such a result is a simple computation. Indeed, the fact that $\widehat{f}$ may be extended to be analytic follows directly from applying Fubini's theorem in conjunction with Morera's theorem, and hence we skip it. For the bound \eqref{eq:paley-wiener}, we start by noting that 
\[
|\widehat{f}(z)| \le \int_{\R} e^{-2\pi \text{Im}(z) t} e^{-\alpha t^2} \, dt.
\]
By completing the squares, it follows at once that 
\[
|\widehat{f}(z)| \le e^{\frac{\pi^2 |\text{Im}(z)|^2}{\alpha}} \int_{\R} e^{-\alpha\left(t + \frac{\pi \text{Im}(z)}{\alpha}\right)^2} \, dt \lesssim C_{\alpha} \cdot e^{\frac{\pi^2 |\text{Im}(z)|^2}{\alpha}}. 
\]
This finishes the proof of the claim. 
\end{proof}
 
 This result directly implies that both $f$ and $g$ may be extended to be entire functions throughout, as well as the same for $\widehat{f}, \widehat{g},$ and that all such functions are, indeed, of \emph{finite type} -- which is precisely the content of \eqref{eq:paley-wiener}. 

In that regard, let then $f$ and $g$ still denote the entire extension of each of those functions. Consider $F(z) = f(z) \overline{f(\overline{z})},$ and define $G$ analogously. By our previous considerations, $F$ and $G$ are two entire functions, each of which is of order $2,$ such that $F(\lambda) = G(\lambda), \, \forall \lambda \in \Lambda.$ 

We are now able to conclude. Since $F-G$ is of order $2,$ if we suppose that it is not identically zero, letting $Z(F-G) := \{z_n\}_{n \in \N}$ denote its zero sequence, we should have, according to \cite[Theorem~2.5.13]{Boas}, that 
\begin{equation}\label{eq:bound-zeros}
\liminf_{r \to \infty} \frac{n_{F-G}(r)}{r^2} < 2 \theta(\alpha),
\end{equation}
where $\theta(\alpha)$ is the constant from Proposition \ref{prop:paley-wiener} above, and $n_{F-G}(r) = \#\{ n \in \N \colon z_n \in B_r(0)\}$ denotes the number of zeros of $F-G$ inside the disk of center 0 and radius $r.$ 

Now, since $\lambda_i \in Z(F-G)$ for each $i \in \N$, by Lemma \ref{lemma:decay-lambda}, we should have, for $r>0$ sufficiently large, 
\begin{equation}\label{eq:density-points}
n_{F-G}(r) \ge 2 \cdot \# \left\{ i \in \N \colon i \gg 1, \sqrt{(2+\varepsilon)/C_{\alpha}} \sqrt{i} < r \right\} \ge \frac{2C_{\alpha} r^2}{2+\varepsilon}. 
\end{equation}
If we choose $C_{\alpha} > \frac{2\pi^2}{\alpha}$, then \eqref{eq:density-points} contradicts \eqref{eq:bound-zeros}, which directly implies that $F \equiv G,$ which, in particular, shows that $|f| \equiv |g|$ on the real line. If we define $\mathfrak{F}(z) = \widehat{f}(z) \overline{\widehat{f}(\overline{z})}$ and $\mathfrak{G}$ in the same way, the same argument shows that $\mathfrak{F} \equiv \mathfrak{G}.$ In other words, we have that $|f| = |g|$ and $|\widehat{f}| = |\widehat{g}|$ holds \emph{pointwise} on the real line, which was our goal.\\

Now, for the proof of Part (II) of Theorem \ref{thm:positive}, we simply recall the following result, which is a direct consequence of the work of Kulikov, Nazarov and Sodin \cite{kns}: 

\begin{theorem*}[Theorem~1(ii) in \cite{kns}] Suppose $\Lambda, \Gamma$ are two discrete subsets of $\R$ satisfying \eqref{eq:weak-density}. Then there is an infinite-dimensional space of functions $f \in \mathcal{S}(\R)$ such that 
\[
f(\lambda) = 0\, \forall \lambda \in \Lambda, \,\,\widehat{f}(\gamma) = 0\, \forall \gamma \in \Gamma. 
\]
\end{theorem*}
With that result, any function in that space which is not equivalently zero agrees with $g \equiv 0$ in a set as in Part (II) of Theorem \ref{thm:positive}, but obviously does not agree in absolute value with it in either space or frequency, concluding the proof of Theorem \ref{thm:positive}.
\end{proof}

\section{Proof of Theorem \ref{thm:discrete-hardy}}

\subsection{Proof of Theorem \ref{thm:discrete-hardy}} We start by noting that Proposition \ref{prop:Schwartz} implies directly that any $f$ as in the statement of Theorem \ref{thm:discrete-hardy} belongs automatically to the Schwartz class. We shall make use of this fact in the continuation of the proof. 

Now, let us again fix $F$ as in \eqref{eq:F-special}. The same use of Proposition \ref{prop:convex} for $f \cdot F_{\xi}$ yields 
\begin{align*}
& \int_{\mathbb{R}}\left(|\xi|-\frac{3}{2} u\right)^{2p}|f|^{2}\left|F_{\xi}\right|^{2} \leqslant \sigma^{-2 \cdot p} \int_{\mathbb{R}}|y|^{2p}\left|\widehat{f} * \widehat{F_{\xi}}\right|^{2} \\
&+ p \left( |\xi| - \frac{3}{2} u\right)^{2p - 1} \sigma^{-1} \sum_{\lambda \in \Lambda} |f(\lambda)|^2 |F_{\xi}(\lambda)|^2,
\end{align*}
where $\sigma \in \left( 0, \frac{C_A}{2}\right)$ is arbitrary. By the same argument as that used in the proof of Proposition \ref{prop:gaussian}, we have 
\begin{align*}
& \left( \frac{K-3}{K}\right)^{2 p}\left(\int_{|x| \geqslant K \cdot u}|x|^{2 p}|f|^{2}\right) \\
& \leqslant \sigma^{-2p} \cdot \int_{\mathbb{R}}|\widehat{f}|^{2}\left(|y|+\frac{k}{\pi u}\right)^{2 p} d y + C \cdot p \sigma^{-1} \cdot \sum_{\lambda \in \Lambda} |\lambda|^{2p-1} |f(\lambda)|^2.
\end{align*}
Thus, for $K>4$ sufficiently large, 
\begin{align}\label{eq:bound-f-fhat-1}
\int|x|^{2 p}|f|^{2} &\leq 2 \cdot \left(\frac{K(K+1)}{K(K-3)\sigma}\right)^{2p}\int_{\R}|\widehat{f}(y)|^{2}|y|^{2 p} d y \cr 
&+(K u)^{2 p} \int|f|^{2}+2\cdot\left( \frac{K}{\sigma(K-3)}\right)^{2p} (K+1)^{2 p}\left(\frac{p}{u}\right)^{2 p} \int|f|^{2} \cr
& + C \cdot p \sigma^{-1} \cdot \left( \frac{K}{K-3}\right)^{2p} \sum_{\lambda \in \Lambda} |\lambda|^{2p-1} |f(\lambda)|^2.
\end{align}
We then fix $K>1$ sufficiently large satisfying the conditions above, and let $u = \delta \cdot p^{\frac{1}{2}},$ where $\delta > 0$ is to be chosen later. We obtain finally that
\begin{align*}
\int |x|^{2sp} |f|^2 & \le \frac{1}{2} \int_{\R} |\widehat{f}(y)|^2 |y|^{2p} \, dy + \left( (\delta \cdot K)^{2p}   + 2 \cdot \left( \frac{K(K+1)}{\delta \sigma(K-3)} \right)^{2p}\right)\cdot p^{p} \int_{\R} |f|^2  \cr 
                    & + C \cdot p \sigma^{-1}  \cdot \left( \frac{K}{K-3}\right)^{2p}  \sum_{\lambda \in \Lambda} |\lambda|^{2p - 1} |f(\lambda)|^2. 
\end{align*}
Now, we swap the roles of $\widehat{f}$ and $f$ in \eqref{eq:bound-f-fhat-1}, and repeat the bulk of the argument above. This implies, in complete analogy to what we did before, that 
\begin{align*}
&\int|x|^{2p}|\widehat{f}|^{2} \leq \frac{1}{2} \int_{\R}|f(y)|^{2}|y|^{2 p} d y 
+\left( (\delta \cdot K)^{2p} + 2 \cdot \left( \frac{K(K+1)}{\delta \sigma(K-3)} \right)^{2p}\right) \cdot p^{p} \int|f|^{2} \cr
& + C \cdot p \sigma^{-1}  \cdot \left( \frac{K}{K-3}\right)^{2p} \sum_{\gamma \in \Gamma} |\gamma|^{2p-1} |\widehat{f}(\gamma)|^2.
\end{align*}
Let $\varepsilon = A-1 > 0$ by hypothesis. We then start fixing our constants: indeed, choose first $K>4$ such that 
$$p \cdot \left( \frac{K}{K-3}\right)^{2p} < (1+\varepsilon/4)^{p}$$
for all $p$ large enough. We then take $\delta > 0$ sufficiently small such that 
$$(\delta \cdot K)^{2p} \cdot p^p < \min\left(\frac{\varepsilon}{100}, \frac{1}{100\varepsilon}\right)^p \cdot \Gamma(p+1),$$
for all $p>0$ sufficiently large. The existence of such $\delta >0$ is ensured by Stirling's formula. Finally, take $\sigma > 0$ such that 
$$
\left( \frac{K(K+1)}{\delta \sigma (K-3)} \right)^{2p} p^p < \min\left(\frac{\varepsilon}{100}, \frac{1}{100\varepsilon}\right)^p \Gamma(p+1),
$$
again for $p$ large. The existence of such a value of $\sigma$ is ensured, once more, by Stirling's formula, as long as we take $C_A$ to be sufficiently large. 

Gathering all that information, and arguing again similarly as in \eqref{eq:upper-bound-discrete}, we have that 
\begin{equation}\label{eq:bound-Hardy-disc}
\int_{\R} |f(x)|^2 |x|^{2p} \, dx \le \min\left(\frac{\varepsilon}{100}, \frac{1}{100\varepsilon}\right)^p \Gamma(p+1) \int_{\R} |f|^2 
+ \, (1+\varepsilon/4)^p \cdot \frac{\Gamma\left( p + 1\right)}{(2A \pi)^{p}}. 
\end{equation}
We shall assume that $p \in \N$ for the remainder of the proof. By rearranging \eqref{eq:bound-Hardy-disc} and summing over all such $p$, we obtain that 
\[
\sum_{p \ge 0} \int_{\R} |f(x)|^2 \left( \left(\frac{1}{(1+\varepsilon/2)\min\left(\frac{\varepsilon}{100}, \frac{1}{100\varepsilon}\right)}\right)^p + \left(\frac{2A \pi}{1+\varepsilon/2}\right)^p \right) \frac{|x|^{2p}}{p!} \, dx < + \infty.
\]
Since $(1+\varepsilon/2) \min\left(\frac{\varepsilon}{100}, \frac{1}{100\varepsilon}\right) < \frac{1}{2\pi}$ for any $\varepsilon>0,$ this readily implies that there is $\eta > 0$ such that 
\[
\int_{\R} |f(x)|^2 e^{ 2(1+\eta)\pi |x|^2} \, dx < +\infty.
\]
Since the exact same argument holds for $\widehat{f},$ we conclude that we must have $f \equiv 0$ by the Cowling-Price uncertainty principle. This concludes the proof of Theorem \ref{thm:discrete-hardy}. 

\section{Higher-dimensional results} 

We next provide, as promised, a proof of Theorems \ref{thm:high-d-gen-pauli} and \ref{thm:discrete-hardy-high-d}, in light of what we did for the one-dimensional case and the results and techniques from \cite{kns} and \cite{Radchenko-Ramos} . 

\subsection{Proof of Theorem \ref{thm:high-d-gen-pauli}} \textbf{Part I.} We start with the followin higher-dimensional version of Lemma \ref{lemma:wirtinger}. 

\begin{claim}\label{claim:poincare} Let $\Omega \subset \R^d$ be the annulus $\Omega = B(0,R) \setminus B(0,r)$. For each $\varepsilon > 0,$ we have 
\[
\|u\|_{L^2(\Omega)}^2 \le (1+\varepsilon)\left( \frac{R}{r}\right)^{d-1} \frac{(R-r)^2}{\pi^2} \|\nabla u\|_{L^2(\Omega)}^2 + (1+\varepsilon^{-1}) R^{d-1} (R -r) \dashint_{\partial \Omega} |u(w)|^2 \, d\mathcal{H}^{d-1} (w). 
\]
\end{claim}

\begin{proof} Consider the one-dimensional function $h_u(t) := \left(\int_{\mathbb{S}^{d-1}} |u(t \cdot w)|^2 \, d \mathcal{H}^{d-1}(w)\right)^{1/2}$. Applying Lemma  \ref{lemma:wirtinger} to this function, we obtain that 
\begin{align*} 
\|u\|_{L^2(\Omega)}^2 & = \omega_d \cdot \int_{\mathbb{S}^{d-1}} \int_r^R |u(t \cdot w)|^2 t^{d-1} \, d \, t \, d \mathcal{H}^{d-1}(w) \cr 
 & = \omega_d R^{d-1} \int_r^R (h_u(t))^2 \, dt \cr 
 & \le \omega_d R^{d-1} \left( \frac{(R-r)^2}{\pi^2} \cdot (1 + \varepsilon) \int_r^R (h_u'(t))^2 \, dt + (1+ \varepsilon^{-1}) (R-r) \left( |h_u(r)|^2 + |h_u(R)|^2 \right) \right).
\end{align*} 
Now, note that 
\[
|h_u'(t)| \le \frac{\int_{\mathbb{S}^{d-1}} |\nabla u (t \cdot w)||u (t \cdot w)| \, d\mathcal{H}^{d-1}(w)}{|h_u(t)|} \le \left( \int_{\mathbb{S}^{d-1}} |\nabla u(t \cdot w)|^2 \, d \mathcal{H}^{d-1} (w) \right)^{1/2},
\]
by Cauchy-Schwarz. Hence, 
\begin{align*}
\|u\|_{L^2(\Omega)}^2 \le \left( \frac{R}{r} \right)^{d-1} \cdot \frac{(R-r)^2}{\pi^2} (1+\varepsilon) \|\nabla u\|_{L^2(\Omega)}^2 + (1+\varepsilon^{-1}) R^{d-1} (R - r ) \dashint_{\partial \Omega} |u(w)|^2 \, d \mathcal{H}^{d-1}(w), 
\end{align*}
which is the desired conclusion. 
\end{proof}

With those results at hand, we now state and prove a higher-dimensional version of Proposition \ref{prop:convex}.  

\begin{lemma}\label{lemma:poincare} Let $t>0$ and let $\varepsilon \ge 1 - \left(1+\frac{1}{2t}\right)^{-\frac{d-1}{2}}.$ Let $f:\R^d \to \R$ be a function such that $f \in L^2(\R^d) \cap H^1(\R^d)$, supported outside of a centered ball of radius $1,$ which vanishes on a collection $\Lambda$ of (centered) spheres with $(1-\varepsilon)(2t)^{-1}$-dense set of radii. Then, for all convex increasing $C^1$-functions $\Phi:\R_+\to\R_+,$ we have 
\[
\Phi(t^2) \int_{\R^d} |f(x)|^2 \, dx \le \int_{\R^d} \Phi(|\xi|^2) |\widehat{f}(\xi)|^2 \, d\xi + C_{\varepsilon} t \cdot \Phi'(t^2) \sum_{S \in \Lambda} \int_{S} |f(u)|^2 \, d\mathcal{H}^{n-1}(u). 
\]
\end{lemma}

\begin{proof} Using Claim \ref{claim:poincare} summed over the spheres at which $f$ vanishes and using Plancherel, in the same way as we did in Lemma \ref{lemma:wirtinger}, we obtain 
\[
\int_{\R^d} |f(x)|^2 \, dx \le (1+\varepsilon) \frac{(1-\varepsilon)^2 }{(2t)^2 \pi ^2} \int_{\R^d} |\nabla f(x)|^2 \, dx + C_{\varepsilon} t^{-1} \sum_{S \in \Lambda} \int_S |f(u)|^2 \, d\mathcal{H}^{d-1} (u). 
\]
\[
\le t^{-2} \int_{\R^d} |y|^2 |\widehat{f}(y)|^2 \, dy + C_{\varepsilon} t^{-1} \sum_{S \in \Lambda} \int_S |f(u)|^2 \, du. 
\]
The result then follows by repeating the proof of Lemma \ref{lemma:wirtinger}. 
\end{proof} 

With that result available to us, we follow the main setup in \cite[Section~5.1]{Radchenko-Ramos}. Indeed, we reduce our set of radii once again by applying Lemma \ref{lemma:sets-lower}. With that reduction, Propositions \ref{prop:schwartz-dec} and \ref{prop:Schwartz} both hold in this context by applying the same methods verbatim. We then let
$$F = 1_{B_u} * \underbrace{\avgI_{B_{u/2k}} * \cdots *  \avgI_{B_{u/2k}}}_{k-\text{fold convolution}} \quad ,$$
where we define $\avgI_B(x) := \frac{1}{|B|} 1_B(x),$ and $B_r$ denotes the Euclidean ball with center $0$ and radius $r$. Using Lemma \ref{lemma:poincare} for $f \cdot F_v$, where $F_v(x) = F(x-v)$, and undertaking the same computations as in \cite[Proof~of~Proposition~5.1]{Radchenko-Ramos}, 
\begin{align*}
\int_{\R^d} |x|^{2p} |f(x)|^2 \, & dx  \le 2^{d} \left( \frac{1}{\sigma} \left(\frac{K}{K-4}\right) \cdot \frac{K+1}{K} \right)^{2p} \int_{\R^d} |\widehat{f}(\xi)|^2 |\xi|^{2p} \, d\xi\cr
 & + \left(K^{2p} u^{2p} + 2^{d+1} \left( \sqrt{\frac{2d + 4}{\pi}} (K+1)\right)^{2p} \left( \frac{p}{u} \right)^{2p} \right) \int_{\R^d} |f(x)|^2 \, dx \cr 
 & + C_d (1+\varepsilon)^p \sum_{S \in \Lambda} r(S)^{2p-1} \int_S |f(w)|^2 \, d \mathcal{H}^{n-1}(w),
\end{align*}
where $K > 4$ is a fixed large parameter, $u > X_0$ will be chosen in a moment, $\sigma > 1$, $\varepsilon>0$ is a fixed small constant depending only on $K$, and $r(S)$ denotes the radius of the sphere $S \in \Lambda$. Naturally, the same holds for $\widehat{f}$ in place of $f$, so that, by employing the same strategy as we did above in the proof of Theorem \ref{thm:positive}, upon optimizing on $u$, we obtain that 
\begin{align*}
\int_{\mathbb{R}^d}|x|^{2 p} & |f(x)|^{2} d x  \leqslant C_d^{p} p^{ p} \int_{\mathbb{R}}|f|^{2} \cr 
& + C_d \cdot (1+\varepsilon)^p \left( \sum_{S \in \Lambda} r(S)^{2p-1} \int_S |f(w)|^2 \, d\mathcal{H}^{d-1}(w) + \sum_{\Sigma \in \Gamma} r(\Sigma)^{2p-1} \int_\Sigma |\widehat{f}(\omega)|^2 \, d\mathcal{H}^{d-1}(\omega) \right).
\end{align*}
Since we have $|f(w)| \le e^{-\alpha|w|^2}, \, |\widehat{f}(\omega)| \le e^{-\alpha |\omega|^2}$, whenever $w \in S \in \Lambda, \, \omega \in \Sigma \in \Gamma$, Lemma \ref{lemma:decay-lambda} implies that the second term on the right-hand side above is bounded by 
\begin{align*} 
C_d \sum_{k \ge 0} & |\lambda_k|^{2p+d-2} e^{-2\alpha|\lambda_k|^2} + \sum_{j \ge 0} |\gamma_j|^{2p+d-2} e^{-2 \alpha |\gamma_j|^2} \cr 
& \le C_d \left( \sum_{k \ge 0} |\lambda_k|^{2p + d -1} |\lambda_{k+1} - \lambda_k| e^{-2\alpha |\lambda_k|^2} + \sum_{j \ge 0} |\gamma_j|^{2p+d-1} |\gamma_{j+1} - \gamma_j| e^{-2\alpha|\gamma_j|^2} \right) \cr 
& \le \rho_d^p \cdot \Gamma\left(p + \frac{d+2}{2}\right),
\end{align*}
for some (possibly dimension-dependent) constant $\rho_d > 0$. Hence, we conclude that 
\[
\int_{\R^d} |x|^{2p} |f(x)|^2 \, dx + \int_{\R^d} |y|^{2p} |\widehat{f}(y)|^2 \, dy \le C_d^p \cdot \left( p^p \|f\|_2^2 + \Gamma\left( p + \frac{d+1}{2} \right)\right). 
\]
We readily conclude, by using once more Stirling's formula and optimizing on $p$, that 
\[
\int_{\R^d} e^{\alpha' |x|^2} |f(x)|^2 \, dx + \int_{\R^d} e^{\alpha' |y|^2} |\widehat{f}(y)|^2 \, dy < + \infty. 
\]
An argument entirely analogous to the one we used in Proposition \ref{prop:paley-wiener} shows that $f$ and $\widehat{f}$ may both be extended to be \emph{entire} on $\C^d$, with the additional property that 
\[
\sup_{z \in \C^d} \left( e^{-\theta(\alpha) |z|^2} |f(z)| + e^{-\theta(\alpha)|z|^2} |\widehat{f}(z)|\right) < +\infty. 
\]
We then consider an axis-parallel line $\ell$ in $\R^d$. By hypothesis, $f|_\ell$ and $g|_\ell$ are both one-dimensional analytic, and they satisfy $\sup_{z \in \C} \left(e^{-\theta(\alpha)|z|^2} |f|_{\ell}(z)| + e^{-\theta(\alpha)|z|^2} |g|_{\ell}(z)|\right) < \infty$. Consider then $h_\ell := (f|_\ell)^* f|_\ell - (g|_{\ell})^* g|_\ell$. This new function is again entire, and it satisfies further that $\sup_{z \in \C} e^{-2\theta(\alpha)|z|^2} |h_\ell(z)| < +\infty$. We then use \cite[Theorem~2.5.13]{Boas} with $h_\ell$, which implies that, since $h_\ell$ vanishes on points which grow at the same asymptotic rate as the radii of spheres in $\Lambda$, we conclude that, under the conditions of our result, we obtain that $h_\ell \equiv 0$, which alternatively reads as $|f|_\ell| = |g|_\ell|$ \emph{pointwise}. Since the axis-parallel line $\ell$ was \emph{arbitrary}, we conclude thus that $|f| = |g|$ holds pointwise everywhere. The same holds obviously on the Fourier side, concluding the proof of Part I of Theorem \ref{thm:high-d-gen-pauli}.

\vspace{2mm}

\noindent\textbf{Part II.} We now indicate how to modify the counterexample constructions from \cite{kns} and \cite{Radchenko-Ramos} in order to prove the second part of Theorem \ref{thm:high-d-gen-pauli}. We start with some preliminaries. 

Let $k:[-\pi,\pi] \to \R$ be defined by 
\begin{equation}\label{eq:k-def-2}
k(\theta) = (\pi \beta) \sin(2 \theta) - \gamma \cos(2 \theta)
\end{equation} whenever $\theta \in [0,\pi/2],$ and extended to $\theta \in [-\pi,\pi]$ such that it becomes even and symmetric with respect to $\pi/2.$ For such $k,$ we call a discrete set $\mathcal{Z} = \mathcal{Z}_1^+ \cup \mathcal{Z}_1^{-} \cup \mathcal{Z}_2^+ \cup \mathcal{Z}_2^{-} \subset \C,$ where $\mathcal{Z}_j^{\pm} \subset \left\{ \arg (z) = \frac{(j\pm1) \pi}{2} \right\},$ \emph{$k-$regular} if:

\begin{enumerate}[(I)]
\item $\mathcal{Z}_j^{\pm}$ has density $m_j/4\pi$ with respect to exponent $2.$ That is, it holds that 
$$|\mathcal{Z}_j^{\pm} \cap (0,r e^{i\theta_j^{\pm}})| \sim r^2m_j/ 4\pi$$
as $r \to \infty$, where $m_1 = 2 \pi \beta, m_2 = 2\gamma,$ with $\theta_j^{\pm}=\frac{(j\pm1) \pi}{2}$;

\vspace{2mm}

\item the disks $\{D_z\}_{z \in \mathcal{Z}} = \{B(z,c\cdot (1+|z|)^{-1})\}_{z \in \mathcal{Z}}$ are all \emph{disjoint} for some $c>0.$  
\end{enumerate} 
\vspace{2mm}

We now state the following result by Levin, which will be crucial in our construction: 

\begin{proposition}[Theorem~5, Ch. II, \cite{Levin}]\label{thm:levin}
Let \( k \) be as in \eqref{eq:k-def-2}. Then, for any \( k \)-regular set \( \mathcal{Z} \), there exists an entire function \( S \), whose zeroes are simple and coincide with \( \mathcal{Z} \), such that, for every \( \varepsilon > 0 \),
\begin{align}
    |S(re^{i\theta})| \le C_{\varepsilon} e^{(k(\theta) + \varepsilon)r^2},  & \text{ everywhere in } \mathbb{C}, \tag{I} \\
    |S(re^{i\theta})| \ge c_{\varepsilon} e^{(k(\theta) - \varepsilon)r^2}, & \text{ whenever } w = re^{i\theta} \not\in \cup_{z \in \mathcal{Z}} D_z. \tag{II}
\end{align}
\end{proposition}

We shall use Proposition \ref{thm:levin} above with $\mathcal{Z}_1^- = - \mathcal{Z}_1^+,$ and $\mathcal{Z}_1^+ = \{ \lambda_k \colon \lambda_k \in \Lambda\}$. We select an arbitrary set $\mathcal{Z}_2^+$ so that $\mathcal{Z}_2^- = - \mathcal{Z}_2^+,$ and the union set $\mathcal{Z} = \cup_{i=1,2} \cup_{\pm} \mathcal{Z}_i^{\pm}$ is $k$-regular for some $\gamma$ to be chosen later. According to \cite[Section~3.3]{Radchenko-Ramos}, we have that the function $S$ which we obtain by that process may be taken to be \emph{even}. 

Consider then $S_k(z) = \frac{S(z)}{z^2 - \lambda_k^2}.$ This function is again even, and by the way we constructed $S$, it is entire and satisfies (by a route application of the maximum principle) 
\begin{equation}\label{eq:S-bound}
|S_{k}(z)| \lesssim C_{\varepsilon} e^{(k(\theta)+\varepsilon)|z|^2}
\end{equation}
in all of $\C$, uniformly on $k \in \N$. 

\begin{proposition} Let $\sigma_k(x) = S_k(|x|)$ for $x \in \R^d$. There exist two numbers $\beta' > \beta'' > 0$ such that 
\[
|\sigma_k(x)| \le C e^{- \beta'' |x|^2}, \quad |\widehat{\sigma_k}(y)| \le C e^{- C \beta'|y|^2}, 
\]
for all $x,y \in \R^d$. 
\end{proposition}

\begin{proof} We start by writing $S_k(z) = H_k(z^2)$, for some entire function $H_k$ with good growth properties. Such a representation clearly holds due to the even nature of the functions $S_k$. We then write 

\begin{align*}
\widehat{\sigma_k}(y) &= \int_{\R^d} S_k(|x|) e^{- 2\pi i x \cdot y} \, dx  = \int_{\R^d} H_k\left( \sum_{j=1}^d x_j^2 \right) e^{- 2 \pi i x \cdot y} \, dx \cr 
                    & = \int_{\R^d} H_k\left( \sum_{j=1}^d (x_j + i t_j)^2 \right) e^{- 2 \pi i (x+ i t) \cdot y} \, dx ,
\end{align*}
where we have used that $|H_k(z)|\le e^{-c|z|}$ for a small cone around the $\R$-axis in order to change the contour of each of the $d$ integrals above, and where $t$ is a multiple of $y$, to be chosen later. We then bound
\begin{align*}
|\widehat{\sigma_k}(y)| &  \le e^{-2\pi |y||t|} \left| \int_{\R^d} \left|H_k \left( \sum_{j=1}^d (x_j +it_j)^2)\right)\right| e^{-\varepsilon r(x,t)} e^{\varepsilon r(x,t)}\, d x \right| \cr 
& \le e^{-2 \pi |y||t|} \sup_{x \in \R^d} \left( \left| H_{k}\left( \sum_{j=1}^d (x_j + i t_j)^2 \right) \right| e^{\varepsilon r(x,t)} \right) \times \left( \int_{\R^d} e^{-\varepsilon r(x,t)} \, dx \right),
\end{align*}
where we write 
\[
\sum_{j=1}^d (x_j + i t_j)^2 = \|x\|^2 - \|t\|^2 + 2 i \langle x,t\rangle =: r(x,t) e^{i \theta(x,t)}. 
\]
We now simply note that $r(x,t) \ge |\|x\|^2 - \|t\|^2|$, which implies that $\int_{\R^d} e^{-\varepsilon r(x,t)} \, dx \lesssim C_{\varepsilon} (1+|t|^{d}).$  Thus, 
\[
|\widehat{\sigma_k}(y)| \lesssim  (1+|t|^d) \cdot e^{-2 \pi |y| |t|}\sup_{x \in \R^d} \left( \left| H_{k}\left( \sum_{j=1}^d (x_j + i t_j)^2 \right) \right| e^{\varepsilon r(x,t)} \right). 
\]
Using then the properties of $S_k$ and reproducing the end of the proof of \cite[Lemma~3.2]{Radchenko-Ramos}, we are able to finish the proof of the claimed inequality. We omit the details and refer the reader to \cite{Radchenko-Ramos,kns} for further details. 
\end{proof}
By now considering $\omega_k(x) = \frac{\sigma_k(x)}{S_k'(\lambda_k)}$, we may conclude that there are $\beta' > \beta > \beta'' > 0$ such that 
\[
|\omega_k(x)| \le C e^{- \beta''|x|^2 + \beta |\lambda_k|^2}, \quad |\widehat{\omega_k}(y)| \le C e^{-\beta' |x|^2 + \beta |\lambda_k|^2},
\]
for each $x,y \in \R^d$. We further have that $\omega_k(\lambda_i) = \delta_{i,k}$. Hence, we are now able to reproduce the iteration process present in the arguments in \cite[Sections~7.4~and~7.5]{kns} and \cite[Section~3.4]{Radchenko-Ramos}. These imply that, if \eqref{eq:weak-density} holds, then there exist an infinite dimensional space of functions $f$ such that $f|_{\Lambda} = \widehat{f}|_{\Gamma} = 0$. In particular, any two of those functions satisfy the assertions of Part II of Theorem \ref{thm:high-d-gen-pauli}, as desired.

\begin{remark} Unlike Theorems \ref{thm:positive} and \ref{thm:discrete-hardy}, the results in Theorems \ref{thm:high-d-gen-pauli} and \ref{thm:discrete-hardy-high-d} are \emph{not discrete}, as we assume that the sets where the absolute values match have both codimension 1. If one is interested, however, in \emph{discrete} results, it seems that the methods in \cite{Ramos-Sousa} are the best known ones at the present moment. 

We note, however, that the \emph{negative} results in \cite{adve} impose natural restrictions - similar to those in Part II of Theorem \ref{thm:positive} - for the automatic Pauli pair property to hold. In a similar spirit, an alternative approach one could possibly employ in order to prove a \emph{positive} result would be to quantify the methods in \cite{adve}, which seems like an interesting problem on its own, but that escapes the scope of this manuscript, for which reason we shall not investigate it here. 
\end{remark}

\subsection{Proof of Theorem \ref{thm:discrete-hardy-high-d}} We recall what we obtained in the proof of Theorem \ref{thm:high-d-gen-pauli} above: indeed, a more careful analysis shows that we actually have
\begin{align*}
\int_{\R^d} |x|^{2p} |f(x)|^2 \, & dx  \le 2^{d} \left( \frac{1}{\sigma} \left(\frac{K}{K-4}\right) \cdot \frac{K+1}{K} \right)^{2p} \int_{\R^d} |\widehat{f}(\xi)|^2 |\xi|^{2p} \, d\xi\cr
 & + \left(K^{2p} u^{2p} + 2^{d+1} \left( 2\sqrt{\frac{d + 2}{\pi}} \frac{K+1}{\sigma} \right)^{2p} \left( \frac{p}{u} \right)^{2p} \right) \int_{\R^d} |f(x)|^2 \, dx \cr 
 & + C_d (1+\varepsilon)^p \sum_{S \in \Lambda} r(S)^{2p-1} \int_S |f(w)|^2 \, d \mathcal{H}^{n-1}(w),
\end{align*}
as long as we choose $K>4$ such that 
$$p \cdot \left( \frac{K}{K-3}\right)^{2p} < (1+\varepsilon/4)^{p}$$
for all $p$ large enough. We once more start fixing parameters: set $\delta > 0$ and choose $u = \delta \cdot p^{1/2}$. We take $\delta > 0$ sufficiently small such that 
$$(\delta \cdot K)^{2p} \cdot p^p < \min\left(\frac{\varepsilon}{100}, \frac{1}{100\varepsilon}\right)^p \cdot \Gamma(p+1),$$
for all $p>0$ sufficiently large. The existence of such $\delta >0$ is ensured by Stirling's formula. Finally, take $\sigma > 0$ such that 
$$
\left( 2 \sqrt{\frac{d+2}{\pi}} \frac{K+1}{\delta \sigma} \right)^{2p} p^p < \min\left(\frac{\varepsilon}{100}, \frac{1}{100\varepsilon}\right)^p \Gamma(p+1),
$$
again for $p$ large. The existence of such a value of $\sigma$ is ensured, once more, by Stirling's formula, as long as we take $C_A$ to be sufficiently large. We obtain hence 
\begin{align*}
\int_{R^d} |x|^{2p} |f(x)|^2 \, dx & + \int_{\R^d} |y|^{2p} |\widehat{f}(y)|^2 \, dy \le  2 \min\left( \frac{\varepsilon}{100}, \frac{1}{100 \varepsilon}\right)^p \Gamma(p+1) \int_{\R^d} |f|^2 & \cr 
 & + \, (1+\varepsilon/2)^p \cdot \frac{\Gamma\left( p + \frac{d}{2} + 1\right)}{\left(2A\pi\right)^{p + \frac{d}{2}}} & \cr 
& \le 2 \min\left( \frac{\varepsilon}{100}, \frac{1}{100 \varepsilon}\right)^p \Gamma(p+1) \int_{\R^d} |f|^2  +\, (1+\varepsilon)^p \frac{\Gamma(p+1)}{(2A\pi)^p}. &
\end{align*}
The remainder of the proof follows the same steps as in the proof of Theorem \ref{thm:discrete-hardy}, and hence we omit it. This concludes the proof of Theorem \ref{thm:discrete-hardy-high-d}.  

\section{Comments and Generalizations}

\subsection{Asymmetric versions of Theorem \ref{thm:positive}} In spite of the fact that the conditions on the discrete sets $\Lambda, \Gamma$ in Theorem \ref{thm:positive} are made to be symmetric, we highlight that this is \emph{not} needed in order to draw similar conclusions. Indeed, the following result showcases that fact: 

\begin{theorem}\label{thm:positive-asym} Fix any $\alpha > 0,$ and two numbers $p,q \in (1,+\infty)$ such that $\frac{1}{p} + \frac{1}{q} = 1.$ Then there is $C_{p,\alpha} > 0$ such that the following holds. Let $\Lambda, \Gamma \subset \R$ be two discrete sets such that 
\begin{equation}
\max\left\{\limsup_{i\to \pm \infty} |\lambda_{i+1} - \lambda_i| |\lambda_i|^{p-1}, \limsup_{i\to \pm \infty} |\gamma_{i+1} - \gamma_i| |\gamma_i|^{q-1}\right\} < \frac{1}{C_{p,\alpha}}.
\end{equation}
Let $f,g:\R \to \C$ be two functions in $H^1 \cap \mathcal{F}(H^1).$ Suppose, moreover, that the function $g$ satisfies 
\[
|g(x)| \lesssim e^{-\alpha |x|^p}, \, \, |\widehat{g}(\xi)| \lesssim e^{-\alpha |\xi|^q},
\]
and that we have $|f(\lambda)| = |g(\lambda)|, \, \forall \, \lambda \in \Lambda, \, \, |\widehat{f}(\gamma)| = |\widehat{g}(\gamma)|, \, \forall \, \gamma \in \Gamma.$ Then it follows that $|f| \equiv |g|$ and $|\widehat{f}| \equiv |\widehat{g}|.$
\end{theorem}

\begin{proof}[Sketch of Proof] First, we remark that the following version of Proposition \ref{prop:gaussian} holds: 

\begin{proposition}\label{prop:extend-decay-general} Suppose $\Lambda=\{\lambda_i\}_{i \in \Z}, \Gamma=\{\gamma_i\}_{i \in \Z}$ are such that
\begin{equation}
\max\left( \limsup_{i \to \pm \infty} |\lambda_{i+1} - \lambda_i| |\lambda_i|^{p-1}, \limsup_{j \to \pm \infty} |\gamma_{j+1} - \gamma_j| |\gamma_j|^{q-1}\right) < \frac{1}{2}.
\end{equation}
Suppose that a function $f:\R \to \mathbb{C}$ satisfies that 
\[
|f(\lambda)| \le C e^{-a |\lambda|^p}, \, \forall \lambda \in \Lambda, \,\, |\widehat{f}(\gamma)| \le C e^{-a|\gamma|^q},\, \, \forall \gamma \in \Gamma. 
\]
Then there is $c_a > 0$ Such that 
\[
|f(x)| \le C e^{-c_a |x|^p} \text{ and } |\widehat{f}(\xi)| \le C e^{-c_a |\xi|^q}, \, \, \forall x, \xi \in \R. 
\]  
\end{proposition}

Since the proof of Proposition \ref{prop:extend-decay-general} is a line-by-line adaptation of the proof of Proposition \ref{prop:gaussian}, we shall leave it as an exercise. 

Now, in order to conclude, we argue as in the proof of Theorem \ref{thm:positive}. Indeed, another, more general version of Proposition \ref{prop:paley-wiener} holds: if $f:\R \to \C$ is a measurable function such that $|f(x)| \lesssim e^{-\alpha |x|^p},\, \forall x \in \R$, for some $\alpha >0$ and some $p>1$, then it follows that $\widehat{f}$ may be extended to be an entire function, which satisfies, moreover, that 
\[
|\widehat{f}(\xi)| \le C e^{C(\alpha)|z|^q},\, \forall \, z \in \C,
\]
where $q = \frac{p}{p-1}.$ We use thus Proposition \ref{prop:extend-decay-general} to a function $f$ as in Theorem \ref{thm:positive-asym}, and then, using the general version of Proposition \ref{prop:paley-wiener} above, we conclude that $f$ and $\widehat{f}$ both may be extended as entire functions, such that 
\[
\sup_{z \in \C} \left( |f(z)| e^{-c |z|^p} + |\widehat{f}(z)| e^{-c |z|^q}\right) < + \infty, \text{ for some } c > 0. 
\]
We let, again, $F,G, \mathfrak{F}, \mathfrak{G}$ be defined in the same way as in the proof of Theorem \ref{thm:positive}. In that way, we conclude that $F,G$ are both entire functions of order $p$, and hence so is their difference. Moreover, we have that $F-G$ vanishes on the set $\Lambda.$ But, similarly as in Lemma \ref{lemma:decay-lambda}, we easily conclude that 
\[
|\lambda_i| \lesssim \frac{1}{C_{p,\alpha}^{1/p}}i^{\frac{1}{p}}, \, |\gamma_i| \lesssim \frac{1}{C_{p,\alpha}^{1/q}} i^{\frac{1}{q}}, \, \text{ for all large enough } i\in \N. 
\]
On the other hand, since $\lambda_i \in Z(F-G), \forall \, i \in \N,$ we infer that, again in light of \cite[Theorem~2.5.13]{Boas}, 
\[
\liminf_{r \to \infty} \frac{n_{F-G}(r)}{r^p} < \tilde{C}_p(\alpha),
\]
for some fixed constant $\tilde{C}_p(\alpha)$. But the bounds we have on $\lambda_i$ imply that 
\[
n_{F-G}(r) \ge 2 \cdot \#\left\{i \in \N \colon i \gg 1, i^{1/p} \lesssim C_{p,\alpha}^{1/p} r \right\} \gtrsim C_{p,\alpha} \cdot r^p.
\]
By making $C_{p,\alpha}$ sufficiently large, we reach a contradiction again. By the exact same argument, we have that $\mathfrak{F} \equiv \mathfrak{G}$, and thus we have concluded Theorem \ref{thm:positive-asym}. 
\end{proof}

In spite of the more general nature of Theorem \ref{thm:positive-asym}, we chose to keep the current version of Theorem \ref{thm:positive} for several reasons, the most prominent of them being avoiding unnecessarily polluting the notation throughout the paper, as well as keeping the exposition streamlined. We note that an asymmetric version of Theorem \ref{thm:high-d-gen-pauli} can also be obtained similarly, but in order to keep the exposition as simple as possible we have opted not to include its formulation here. 

\subsection{Sharp versions of Theorem \ref{thm:positive}} As highlighted by Part (II) of Theorem \ref{thm:positive}, we observe that, as long as the density of the sets $\Lambda, \Gamma$ which appear in Theorem \ref{thm:positive} surpass a certain threshold, that result no longer holds. 

In both of such instances, we are measuring density roughly in terms of the best power $\alpha$ such that $\lambda_i \approx i^{\alpha}$ for $i$ sufficiently large, in a suitable sense. In light of that, we pose the following question: 
\begin{question}
For each $\alpha>0$, what is the best constant $C_{\alpha} > 0$ such that, whenever
    \[
    \max\left( \limsup_{i \to \pm \infty} |\lambda_i - \lambda_{i+1}| |\lambda_i| , \limsup_{i \to \pm \infty} |\gamma_i - \gamma_{i+1}| |\gamma_i| \right) < \frac{1}{C_\alpha},
    \]
    then we automatically have that the conclusion of Theorem \ref{thm:positive} holds?     
\end{question}

At this moment, that question seems to be require methods beyond the ones we employed - for instance, if we assume both $f$ and $g$ have space and frequency decay of the kind $e^{-\alpha|x|^2}, $ Theorem \ref{thm:positive} still gives us the bound $C_{\alpha} > \frac{2\pi^2}{ \alpha},$ which is not sharp as $\alpha \to \pi$. Indeed, the proof of Part (II) of Theorem \ref{thm:positive} shows that if 

 \[
 \max\left( \limsup_{i \to \pm \infty} |\lambda_i - \lambda_{i+1}| |\lambda_i| , \limsup_{i \to \pm \infty} |\gamma_i - \gamma_{i+1}| |\gamma_i| \right) > \frac{1}{2},
 \]
 then there are non-zero functions vanishing on $\Lambda$, with $\widehat{f}$ vanishing on $\Gamma$. Since those counterexamples only work for sufficiently small $\alpha > 0$, it would be interesting to investigate the behaviour of the constant $C_{\alpha}$ in that regime. 

 \subsection{Sharpness of Theorem \ref{thm:discrete-hardy}} We remark how sharp Theorem \ref{thm:discrete-hardy} is. Indeed, first of all, note that if we replace \eqref{eq:upper-bound-density-hardy} by the stronger assumption that there is $s>1$ such that 
 \begin{equation}\label{eq:denser-sequences}
 \max\left( \limsup_{j \to \pm \infty}  |\lambda_j - \lambda_{j+1}| |\lambda_j|^s, \limsup_{j \to \pm \infty} |\gamma_j - \gamma_{j+1}||\gamma_j|^s \right) < +\infty,
 \end{equation}
 then an adaptation of its proof shows that the same result holds. That is, the rate of growth of the sequences allowed in Theorem \ref{thm:discrete-hardy} is sharp. 
 
 Moreover, if $A=1,$ we have the classical example of $f(x) = e^{-\pi |x|^2},$ which shows that the claim is false in that case, even under the stronger condition \eqref{eq:denser-sequences}. Furthermore, we note that if the assumptions on $\Lambda, \Gamma$ are less restrictive - as is the case if, for instance, we take $C_A < 2$ in \eqref{eq:upper-bound-density-hardy} - then the conclusion also fails, since the counterexamples from Part (II) of Theorem \ref{thm:positive} apply in this context once more. \\

With that in mind, we highlight the following question:
\begin{question} Regarding Theorem \ref{thm:discrete-hardy}:
\begin{enumerate}
    \item[(I)] If $A=1$, is there a constant $C_1 > 0$ such that, if $\Lambda,\Gamma$ satisfy  
     \begin{equation}\label{eq:denser-sequences-2}
 \max\left( \limsup_{j \to \pm \infty}  |\lambda_j - \lambda_{j+1}| |\lambda_j|, \limsup_{j \to \pm \infty} |\gamma_j - \gamma_{j+1}||\gamma_j| \right) < \frac{1}{C_1},
 \end{equation}
 do we still have that the unique function $f$ such that 
 \[
 |f(\lambda_i)| \lesssim e^{-\pi |\lambda_i|^2}, \, \forall \, i \in \Z, \, |\widehat{f}(\gamma_i)| \lesssim e^{-\pi |\gamma_i|^2}, \, \forall \, i \in \Z
 \]
 is $f(x) = c \cdot e^{-\pi |x|^2}$ for some $c \in \C$? 
    \item[(II)] What is the best constant $C_A>0$ such that if $\Lambda,\Gamma$ satisfy \eqref{eq:upper-bound-density-hardy}, then we can still conclude that $f \equiv 0$ in Theorem \ref{thm:discrete-hardy}? 
\end{enumerate}
\end{question}
We currently do not have any conjecture on the answer of such questions. As a matter of fact, Part (I) would likely need a technique for proving estimates on such functions which is \emph{not} $L^2$-based, as one would imagine that the classical version of Hardy's result would be useful. 

For Part (II) above, we quickly remark on the estimates for $C_A$ that the proof of Theorem \ref{thm:discrete-hardy} yields. A quick scan shows that we need to take $K \sim \frac{1}{A-1}$ in order to run our argument, which implies $\delta \sim (A-1)^{3/2},$ and forces us to choose $C_A \sim (A-1)^{-3}.$ In light of that, it would be already of interest to investigate the optimal value of $\beta > 0$ for which $\sup_{A \text{ close to } 1} (A-1)^{\beta} \cdot \mathfrak{C}_A < +\infty$, where $\mathfrak{C}_A$ is the optimal constant for Part (II).  

\bibliography{cites} 

\providecommand{\bysame}{\leavevmode\hbox to3em{\hrulefill}\thinspace}
\providecommand{\MR}{\relax\ifhmode\unskip\space\fi MR }
\providecommand{\MRhref}[2]{%
  \href{http://www.ams.org/mathscinet-getitem?mr=#1}{#2}
}
\providecommand{\href}[2]{#2}
\begin{thebibliography}{10}

\bibitem{adve}
Anshul Adve, \emph{Density criteria for {F}ourier uniqueness phenomena in
  $\mathbf{R}^d$}, arXiv preprint arXiv:2306.07475 (2023).

\bibitem{Drenth}
A.M.J.~Huiser A.J.J.~Drenth and H.A. Ferwerda, \emph{The problem of phase
  retrieval in light and electron microscopy of strong objects}, Optica Acta:
  International Journal of Optics \textbf{22} (1975), no.~7, 615--628.

\bibitem{alaifari-daubechies}
Rima Alaifari, Ingrid Daubechies, Philipp Grohs, and Rujie Yin, \emph{Stable
  phase retrieval in infinite dimensions}, Found. Comput. Math. \textbf{19}
  (2019), no.~4, 869--900. \MR{3989716}

\bibitem{allman}
B.~E. Allman, P.~J. McMahon, Keith~Alexander Nugent, D.~Paganin, David~L.
  Jacobson, Muhammad Arif, and S.~A. Werner, \emph{Phase radiography with
  neutrons}, nature \textbf{408} (2000), no.~6809, 158--159.

\bibitem{Amrein-Berthier}
W.~O. Amrein and A.~M. Berthier, \emph{On support properties of
  {$L\sp{p}$}-functions and their {F}ourier transforms}, J. Functional Analysis
  \textbf{24} (1977), no.~3, 258--267. \MR{461025}

\bibitem{Bakan}
Andrew Bakan, Haakan Hedenmalm, Alfonso Montes-Rodr\'iguez, Danylo Radchenko,
  and Maryna Viazovska, \emph{Fourier uniqueness in even dimensions}, Proc.
  Natl. Acad. Sci. USA \textbf{118} (2021), no.~15, Paper No. 2023227118, 4.
  \MR{4294062}

\bibitem{balan}
Radu Balan, Pete Casazza, and Dan Edidin, \emph{On signal reconstruction
  without phase}, Appl. Comput. Harmon. Anal. \textbf{20} (2006), no.~3,
  345--356. \MR{2224902}

\bibitem{bandeira-pr}
Afonso~S. Bandeira, Jameson Cahill, Dustin~G. Mixon, and Aaron~A. Nelson,
  \emph{Saving phase: injectivity and stability for phase retrieval}, Appl.
  Comput. Harmon. Anal. \textbf{37} (2014), no.~1, 106--125. \MR{3202304}

\bibitem{Boas}
Ralph~Philip Boas, Jr., \emph{Entire functions}, Academic Press, Inc., New
  York, 1954. \MR{68627}

\bibitem{brs}
Andriy Bondarenko, Danylo Radchenko, and Kristian Seip, \emph{Fourier
  interpolation with zeros of zeta and {$L$}-functions}, Constr. Approx.
  \textbf{57} (2023), no.~2, 405--461. \MR{4577389}

\bibitem{BCK}
Jean Bourgain, Laurent Clozel, and Jean-Pierre Kahane, \emph{Principe
  d'{H}eisenberg et fonctions positives}, Ann. Inst. Fourier (Grenoble)
  \textbf{60} (2010), no.~4, 1215--1232. \MR{2722239}

\bibitem{candes-eldar}
Emmanuel~J. Cand\`es, Yonina~C. Eldar, Thomas Strohmer, and Vladislav
  Voroninski, \emph{Phase retrieval via matrix completion [reprint of
  {MR}3032952]}, SIAM Rev. \textbf{57} (2015), no.~2, 225--251. \MR{3345342}

\bibitem{candes-xiaodong}
Emmanuel~J. Cand\`es and Xiaodong Li, \emph{Solving quadratic equations via
  {P}hase{L}ift when there are about as many equations as unknowns}, Found.
  Comput. Math. \textbf{14} (2014), no.~5, 1017--1026. \MR{3260258}

\bibitem{Christ-Mitch-Ben}
Michael Christ, Ben Pineau, and Mitchell~A. Taylor, \emph{Examples of
  h{\"o}lder-stable phase retrieval}, arXiv preprint arXiv:2205.00187 (2022).

\bibitem{cohn-sphere}
Henry Cohn, \emph{From sphere packing to {F}ourier interpolation}, Bull. Amer.
  Math. Soc. (N.S.) \textbf{61} (2024), no.~1, 3--22. \MR{4678569}

\bibitem{Cohn-Goncalves}
Henry Cohn and Felipe Gon\c{c}alves, \emph{An optimal uncertainty principle in
  twelve dimensions via modular forms}, Invent. Math. \textbf{217} (2019),
  no.~3, 799--831. \MR{3989254}

\bibitem{Viazovska-24}
Henry Cohn, Abhinav Kumar, Stephen~D. Miller, Danylo Radchenko, and Maryna
  Viazovska, \emph{The sphere packing problem in dimension 24}, Ann. of Math.
  (2) \textbf{185} (2017), no.~3, 1017--1033. \MR{3664817}

\bibitem{ckmrv-universal}
\bysame, \emph{Universal optimality of the {$E_8$} and {L}eech lattices and
  interpolation formulas}, Ann. of Math. (2) \textbf{196} (2022), no.~3,
  983--1082. \MR{4502595}

\bibitem{corbett}
J.~V. Corbett and C.~A. Hurst, \emph{Are wave functions uniquely determined by
  their position and momentum distributions?}, The ANZIAM Journal \textbf{20}
  (1977), no.~2, 182--201.

\bibitem{Cowling-Price}
Michael Cowling and John~F. Price, \emph{Generalisations of heisenberg's
  inequality}, Harmonic Analysis: Proceedings of a Conference Held in Cortona,
  Italy, July 1--9, 1982, Springer, 2006, pp.~443--449.

\bibitem{eldar-mendelson}
Yonina~C. Eldar and Shahar Mendelson, \emph{Phase retrieval: stability and
  recovery guarantees}, Appl. Comput. Harmon. Anal. \textbf{36} (2014), no.~3,
  473--494. \MR{3175089}

\bibitem{EKPV2}
L.~Escauriaza, C.~E. Kenig, G.~Ponce, and L.~Vega, \emph{Hardy's uncertainty
  principle, convexity and {S}chr\"odinger evolutions}, J. Eur. Math. Soc.
  (JEMS) \textbf{10} (2008), no.~4, 883--907. \MR{2443923}

\bibitem{EKPV1}
Luis Escauriaza, Carlos~E. Kenig, Gustavo Ponce, and Luis Vega, \emph{The sharp
  {H}ardy uncertainty principle for {S}chr\"odinger evolutions}, Duke Math. J.
  \textbf{155} (2010), no.~1, 163--187. \MR{2730375}

\bibitem{feigenbaum}
Ahram~S. Feigenbaum, Peter~J. Grabner, and Douglas~P. Hardin,
  \emph{Eigenfunctions of the {F}ourier transform with specified zeros}, Math.
  Proc. Cambridge Philos. Soc. \textbf{171} (2021), no.~2, 329--367.
  \MR{4299592}

\bibitem{folland}
Gerald~B. Folland and Alladi Sitaram, \emph{The uncertainty principle: a
  mathematical survey}, Journal of Fourier analysis and applications \textbf{3}
  (1997), 207--238.

\bibitem{Mitchell-and-friends}
D.~Freeman, T.~Oikhberg, B.~Pineau, and M.~A. Taylor, \emph{Stable phase
  retrieval in function spaces}, Mathematische Annalen (2024), Online first
  version.

\bibitem{goncalves-ramos}
Felipe Gon\c{c}alves and Jo\~ao P.~G. Ramos, \emph{A note on discrete
  {H}eisenberg uniqueness pairs for the parabola}, Bull. Sci. Math.
  \textbf{174} (2022), Paper No. 103095, 9. \MR{4362794}

\bibitem{grohs-liehr}
Philipp Grohs and Lukas Liehr, \emph{Injectivity of {G}abor phase retrieval
  from lattice measurements}, Appl. Comput. Harmon. Anal. \textbf{62} (2023),
  173--193. \MR{4484790}

\bibitem{grohs-rathmair-2}
Philipp Grohs and Martin Rathmair, \emph{Stable {G}abor phase retrieval and
  spectral clustering}, Comm. Pure Appl. Math. \textbf{72} (2019), no.~5,
  981--1043. \MR{3935477}

\bibitem{grohs-rathmair}
\bysame, \emph{Stable {G}abor phase retrieval for multivariate functions}, J.
  Eur. Math. Soc. (JEMS) \textbf{24} (2022), no.~5, 1593--1615. \MR{4404785}

\bibitem{Hardy}
G.~H. Hardy, \emph{A {T}heorem {C}oncerning {F}ourier {T}ransforms}, J. London
  Math. Soc. \textbf{8} (1933), no.~3, 227--231. \MR{1574130}

\bibitem{Hedenmalm-Montes-1}
Haakan Hedenmalm and Alfonso Montes-Rodr\'iguez, \emph{Heisenberg uniqueness
  pairs and the {K}lein-{G}ordon equation}, Ann. of Math. (2) \textbf{173}
  (2011), no.~3, 1507--1527. \MR{2800719}

\bibitem{Hedenmalm-Montes-2}
\bysame, \emph{The {K}lein-{G}ordon equation, the {H}ilbert transform, and
  dynamics of {G}auss-type maps}, J. Eur. Math. Soc. (JEMS) \textbf{22} (2020),
  no.~6, 1703--1757. \MR{4092897}

\bibitem{Hedenmalm-Montes-3}
\bysame, \emph{The {K}lein-{G}ordon equation, the {H}ilbert transform and
  {G}auss-type maps: {$H^{\infty}$} approximation}, J. Anal. Math. \textbf{144}
  (2021), no.~1, 119--190. \MR{4361892}

\bibitem{hogan-lakey}
Jeffrey~A. Hogan and Joseph~D. Lakey, \emph{Time-frequency and time-scale
  methods: adaptive decompositions, uncertainty principles, and sampling},
  Springer, 2005.

\bibitem{hurt}
Norman~E. Hurt, \emph{Phase retrieval and zero crossings: mathematical methods
  in image reconstruction}, vol.~52, Springer Science \& Business Media, 2001.

\bibitem{ismagilov}
Rais~Sal'manovich Ismagilov, \emph{On the pauli problem}, Functional Analysis
  and its Applications \textbf{30} (1996), no.~2, 138--140.

\bibitem{jaming-phase}
Philippe Jaming, \emph{Phase retrieval techniques for radar ambiguity
  problems}, Journal of Fourier Analysis and Applications \textbf{5} (1999),
  309--329.

\bibitem{jaming-liehr}
Philippe Jaming and Martin Rathmair, \emph{Gabor phase retrieval via
  semidefinite programming}, arXiv preprint arXiv:2310.11214 (2023).

\bibitem{kehle-ramos}
Christoph Kehle and Jo\~ao P.~G. Ramos, \emph{Uniqueness of solutions to
  nonlinear {S}chr\"odinger equations from their zeros}, Ann. PDE \textbf{8}
  (2022), no.~2, Paper No. 21, 36. \MR{4486177}

\bibitem{klibanov}
Michael~V. Klibanov, Paul~E. Sacks, and Alexander~V. Tikhonravov, \emph{The
  phase retrieval problem}, Inverse problems \textbf{11} (1995), no.~1, 1.

\bibitem{Kulikov-alone}
Aleksei Kulikov, \emph{Fourier interpolation and time-frequency localization},
  J. Fourier Anal. Appl. \textbf{27} (2021), no.~3, Paper No. 58, 8.
  \MR{4273648}

\bibitem{kns}
Aleksei Kulikov, Fedor Nazarov, and Mikhail Sodin, \emph{Fourier uniqueness and
  non-uniqueness pairs}, arXiv preprint arXiv:2306.14013 (2023).

\bibitem{Levin}
B.~Ya. Levin, \emph{Lectures on entire functions}, vol. 150, American
  Mathematical Soc., 1996.

\bibitem{millane}
Rick~P. Millane, \emph{Phase retrieval in crystallography and optics}, JOSA A
  \textbf{7} (1990), no.~3, 394--411.

\bibitem{Radchenko-Ramos}
Danylo Radchenko and Jo{\~a}o P.~G. Ramos, \emph{Perturbed lattice crosses and
  heisenberg uniqueness pairs}, arXiv preprint arXiv:2410.04557 (2024).

\bibitem{radchenko-stoller}
Danylo Radchenko and Martin Stoller, \emph{Fourier non-uniqueness sets from
  totally real number fields}, Comment. Math. Helv. \textbf{97} (2022), no.~3,
  513--553. \MR{4468993}

\bibitem{Radchenko-Viazovska}
Danylo Radchenko and Maryna Viazovska, \emph{Fourier interpolation on the real
  line}, Publ. Math. Inst. Hautes \'Etudes Sci. \textbf{129} (2019), 51--81.
  \MR{3949027}

\bibitem{Ramos-Sousa}
Jo\~ao P.~G. Ramos and Mateus Sousa, \emph{Fourier uniqueness pairs of powers
  of integers}, J. Eur. Math. Soc. (JEMS) \textbf{24} (2022), no.~12,
  4327--4351. \MR{4493626}

\bibitem{ramos-sousa-2}
\bysame, \emph{Perturbed interpolation formulae and applications}, Anal. PDE
  \textbf{16} (2023), no.~10, 2327--2384. \MR{4678143}

\bibitem{ramos-stoller}
Jo\~ao P.~G. Ramos and Martin Stoller, \emph{Perturbed {F}ourier uniqueness and
  interpolation results in higher dimensions}, J. Funct. Anal. \textbf{282}
  (2022), no.~12, Paper No. 109448, 34. \MR{4403065}

\bibitem{reichenbach}
Hans Reichenbach, \emph{Philosophic foundations of quantum mechanics}, Courier
  Corporation, 1998.

\bibitem{Rosenblatt}
Joseph Rosenblatt, \emph{Phase retrieval}, Comm. Math. Phys. \textbf{95}
  (1984), no.~3, 317--343. \MR{765273}

\bibitem{seelamantula}
Chandra~Sekhar Seelamantula, Martin~L Villiger, Rainer~A Leitgeb, and Michael
  Unser, \emph{Exact and efficient signal reconstruction in frequency-domain
  optical-coherence tomography}, JOSA A \textbf{25} (2008), no.~7, 1762--1771.

\bibitem{stoller}
Martin Stoller, \emph{Fourier interpolation from spheres}, Trans. Amer. Math.
  Soc. \textbf{374} (2021), no.~11, 8045--8079. \MR{4328691}

\bibitem{Viazovska-8}
Maryna Viazovska, \emph{The sphere packing problem in dimension 8}, Ann. of
  Math. (2) \textbf{185} (2017), no.~3, 991--1015. \MR{3664816}

\end{thebibliography}
\bibliographystyle{amsplain}

\end{document}